\documentclass{amsart}
\usepackage{amssymb,amsmath}
\usepackage{graphicx}
\theoremstyle{plain}
\newtheorem{thm}{Theorem}[section]

\newtheorem{cor}[thm]{Corollary}
\newtheorem{lem}[thm]{Lemma}
\newtheorem{prop}[thm]{Proposition}

\theoremstyle{definition}
\newtheorem{defn}{Definition}[section]

\theoremstyle{remark}
\newtheorem{rem}{Remark}[section]
\newtheorem{ex}{Example}[section]

\numberwithin{equation}{section}

\newcommand{\ra}{\rightarrow}

\begin{document}

\title{On the mixing time and spectral gap for birth and death chains}

\author[G.-Y. Chen]{Guan-Yu Chen$^1$}

\author[L. Saloff-Coste]{Laurent Saloff-Coste$^2$}

\thanks{$^1$Partially supported by NSC grant NSC100-2115-M-009-003-MY2}

\address{$^1$Department of Applied Mathematics, National Chiao Tung University, Hsinchu 300, Taiwan}
\email{gychen@math.nctu.edu.tw}

\thanks{$^2$Partially supported by NSF grant DMS-1004771}

\address{$^2$Malott Hall, Department of Mathematics, Cornell University, Ithaca, NY 14853-4201}
\email{lsc@math.cornell.edu}



\begin{abstract}
For birth and death chains, we derive bounds on the spectral gap and mixing time in terms of birth and death rates. Together with the results of Ding {\it et al.} in \cite{DLP10}, this provides a criterion for the existence of a cutoff in terms of the birth and death rates. A variety of illustrative examples are treated.
\end{abstract}

\maketitle

\section{Introduction}\label{s-intro}
Let $\Omega$ be a countable set and $(\Omega,K,\pi)$ be an irreducible Markov chain on $\Omega$ with transition matrix $K$ and stationary distribution $\pi$. Let $I$ be the identity matrix indexed by $\Omega$ and
\[
 H_t=e^{-t(I-K)}=\sum_{i=0}^\infty e^{-t}t^iK^i/i!
\]
be the associated semigroup which describes the corresponding natural continuous time process on $\Omega$. For $\delta\in(0,1)$, set
\begin{equation}\label{eq-lazyK}
    K_\delta=\delta I+(1-\delta) K.
\end{equation}
Clearly, $K_\delta$ is similar to $K$ but with an additional holding probability depending of $\delta$. We call $K_\delta$ the $\delta$-lazy walk or $\delta$-lazy chain of $K$. It is well-known that if $K$ is irreducible with stationary distribution $\pi$, then
\[
    \lim_{m\ra\infty}K_\delta^m(x,y)=\lim_{t\ra\infty}H_t(x,y)=\pi(y),\quad\forall  x,y\in\Omega,\,\delta\in(0,1).
\]

In this paper, we consider convergence in total variation. The total variation between two probabilities $\mu,\nu$ on $\Omega$ is defined by $\|\mu-\nu\|_{\text{\tiny TV}}=\sup\{\mu(A)-\nu(A)|A\subset\Omega\}$. For any irreducible $K$ with stationary distribution $\pi$, the (maximum) total variation distance is defined by
\begin{equation}\label{eq-tv}
    d_{\text{\tiny TV}}(m)=\sup_{x\in\Omega}\|K^m(x,\cdot)-\pi\|_{\text{\tiny TV}},
\end{equation}
and the corresponding mixing time is given by
\begin{equation}\label{eq-tvmix}
    T_{\text{\tiny TV}}(\epsilon)=\inf\{m\ge 0|d_{\text{\tiny TV}}(m)\le\epsilon\},\quad\forall \epsilon\in(0,1).
\end{equation}
We write $d_{\text{\tiny TV}}^{(c)},T_{\text{\tiny TV}}^{(c)}$ for the total variation distance and mixing time for the continuous semigroup and $d_{\text{\tiny TV}}^{(\delta)},T_{\text{\tiny TV}}^{(\delta)}$ for the $\delta$-lazy walk.

A birth and death chain on $\{0,1,...,n\}$ with birth rate $p_i$, death rate $q_i$ and holding rate $r_i$ is a Markov chain with transition matrix $K$ given by
\[
 K(i,i+1)=p_i,\quad K(i,i-1)=q_i,\quad K(i,i)=r_i,\quad\forall 0\le i\le n,
\]
where $p_i+q_i+r_i=1$ and $p_n=q_0=0$. It is obvious that $K$ is irreducible if and only if $p_iq_{i+1}>0$ for $0\le i<n$. Under the assumption of irreducibility, the unique stationary distribution $\pi$ of $K$ is given by $\pi(i)=c(p_0\cdots p_{i-1})/(q_1\cdots q_i)$, where $c$ is a positive constant such that $\sum_{i=0}^n\pi(i)=1$. The following theorem provides a bound on the mixing time using the birth and death rates and is treated in Theorems \ref{t-upper} and \ref{t-lower}.

\begin{thm}\label{t-mixingtime}
Let $K$ be an irreducible birth and death chain on $\{0,1,...,n\}$ with birth, death and holding rates $p_i,q_i,r_i$. Let $i_0$ be a state satisfying $\pi([0,i_0])\ge 1/2$ and $\pi([i_0,n])\ge 1/2$, where $\pi(A)=\sum_{i\in A}\pi(i)$, and set
\[
 t=\max\left\{\sum_{k=0}^{i_0-1}\frac{\pi([0,k])}{\pi(k)p_k},\sum_{k=i_0+1}^n\frac{\pi([k,n])}{\pi(k)q_k}\right\}.
\]
Then, for any $\delta\in[1/2,1)$,
\[
 \min\left\{T_{\textnormal{\tiny TV}}^{(c)}(1/10),T_{\textnormal{\tiny TV}}^{(\delta)}(1/20)\right\}\ge \frac{t}{6},
\]
and
\[
 \max\left\{T_{\textnormal{\tiny TV}}^{(c)}(\epsilon),T_{\textnormal{\tiny TV}}^{(\delta)}(\epsilon)\right\}\le\frac{18t}{\epsilon^2},\quad\forall\epsilon\in(0,1).
\]
\end{thm}
The authors of \cite{DLP10} derive a similar upper bound. Note that if $(X_m)_{m=0}^\infty$ is a Markov chain on $\Omega_n$ with transition matrix $K$ and $\tau_i:=\min\{m\ge 0|X_m=i\}$, then $t=\max\{\mathbb{E}_0\tau_{i_0},\mathbb{E}_n\tau_{i_0}\}$, where $\mathbb{E}_i$ denotes the conditional expectation given $X_0=i$. See Lemma \ref{l-passage} for details.

A sharp transition phenomenon, known as cutoff, was observed by Aldous and Diaconis in early 1980s. See e.g. \cite{D96cutoff,CSal08} for an introduction and a general review of cutoffs. In total variation, a family of irreducible Markov chains $(\Omega_n,K_n,\pi_n)_{n=1}^\infty$ is said to present a cutoff if
\begin{equation}\label{eq-cutmix}
    \lim_{n\ra\infty}\frac{T_{n,\text{\tiny TV}}(\epsilon)}{T_{n,\text{\tiny TV}}(\eta)}=1,\quad\forall 0<\epsilon<\eta<1.
\end{equation}
The family is said to present a $(t_n,b_n)$ cutoff if $b_n=o(t_n)$ and
\[
 |T_{n,\text{\tiny TV}}(\epsilon)-t_n|=O(b_n),\quad\forall 0<\epsilon<1.
\]
The cutoff for the associated continuous semigroups is defined in a similar way. Given a family $\mathcal{F}$ of irreducible Markov chains, we write $\mathcal{F}_c$ and $\mathcal{F}_\delta$ for the families of corresponding continuous time chain and $\delta$-lazy discrete time chains.

Let $\mathcal{F}=\{(\Omega_n,K_n,\pi_n)|n=1,2,...\}$ be a family of birth and death chains, where $\Omega_n=\{0,1,...,n\}$ and $K_n$ has birth rate $p_{n,i}$, death rate $q_{n,i}$ and holding rate $r_{n,i}$. Suppose that $K_n$ is irreducible with stationary distribution $\pi_n$. For the family $\{(\Omega_n,K_n,\pi_n)|n=1,2,...\}$, Ding {\it et al.} \cite{DLP10} showed that, in the discrete time case and assuming $\inf_{i,n}r_{n,i}>0$, the cutoff in total variation exists if and only if the product of the total variation mixing time and the spectral gap, i.e. the smallest non-zero eigenvalue of $I-K$, tends to infinity. There is also a similar version for the continuous time case.
In \cite{CSal12-2}, we use the results of \cite{DS06,DLP10} to provide another criterion on the cutoff using the eigenvalues of $K_n$. In both cases, the spectral gap is needed to determine if there is a cutoff. The following theorem provides a bound on the spectral gap using the birth and death rates.

\begin{thm}\label{t-spectralgap}
Consider an irreducible birth and death chain $K$ on $\{0,1,...,n\}$ with birth, death and holding rates, $p_i,q_i,r_i$. Let $\pi$ and $\lambda$ be the stationary distribution and spectral gap of $K$ and set
\[
 \ell=\max\left\{\max_{j:j<i_0}\sum_{k=j}^{i_0-1}\frac{\pi([0,j])}{\pi(k)p_k},\max_{j:j> i_0}\sum_{k=i_0+1}^j\frac{\pi([j,n])}{\pi(k)q_k}\right\},
\]
where $i_0$ is a state such that $\pi([0,i_0])\ge 1/2$ and $\pi([i_0,n])\ge 1/2$. Then,
\[
 \frac{1}{4\ell}\le\lambda\le\frac{2}{\ell}.
\]
\end{thm}
The above theorem is motivated by \cite{M99}, where the author considers the spectral gap of birth and death chains on $\mathbb{Z}$. We refer the reader to \cite{M99} and the references therein for more information. Note that if $t,\ell$ are the constants in Theorem \ref{t-mixingtime}-\ref{t-spectralgap}, then $t\ge\ell$. Based on the results in \cite{DLP10}, we obtain a theorem regarding cutoffs for birth and death chains.

\begin{thm}\label{t-bdc-cut-main}
Consider a family of irreducible birth and death chains
\[
 \mathcal{F}=\{(\Omega_n,K_n,\pi_n)|n=1,2,...\},
\]
where $\Omega_n=\{0,1,...,n\}$ and $K_n$ has birth, death and holding rates, $p_{n,i},q_{n,i},r_{n,i}$. For $n\ge 1$, let $i_n\in\{0,...,n\}$ be a state satisfying $\pi_n([0,i_n])\ge 1/2$ and $\pi_n([i_n+1,n])\ge 1/2$ and set
\[
 t_n=\max\left\{\sum_{k=0}^{i_n-1}\frac{\pi_n([0,k])}{\pi_n(k)p_{n,k}},\sum_{k=i_n+1}^n\frac{\pi_n([k,n])}{\pi_n(k)q_{n,k}}\right\}.
\]
and
\[
 \ell_n=\max\left\{\max_{j:j<i_n}\sum_{k=j}^{i_n-1}\frac{\pi_n([0,j])}{\pi_n(k)p_{n,k}},\max_{j:j>i_n}\sum_{k=i_n+1}^j\frac{\pi_n([j,n])}{\pi_n(k)q_{n,k}}\right\},
\]
Then, for any $\epsilon\in(0,1/2)$ and $\delta\in(0,1)$, there is a constant $C=C(\epsilon,\delta)>1$ such that
\[
 C^{-1}t_n\le \min\{T_{n,\textnormal{\tiny TV}}^{(c)}(\epsilon),T_{n,\textnormal{\tiny TV}}^{(\delta)}(\epsilon)\}\le\max\{T_{n,\textnormal{\tiny TV}}^{(c)}(\epsilon),T_{n,\textnormal{\tiny TV}}^{(\delta)}(\epsilon)\}\le Ct_n,
\]
for $n$ large enough. Moreover, the following are equivalent.
\begin{itemize}
\item[(1)] $\mathcal{F}_c$ has a total variation cutoff.

\item[(2)] For $\delta\in(0,1)$, $\mathcal{F}_\delta$ has a total variation cutoff.

\item[(3)] $t_n\ell_n\ra\infty$.
\end{itemize}
\end{thm}

The above theorem is immediate from Theorems \ref{t-mixingtime}, \ref{t-spectralgap}, \ref{t-csal12-2} and \ref{t-mixingcomp}. The selection of $i_n$ can be relaxed. See Theorem \ref{t-tvmixingcn} for a precise statement. By the results in \cite{CSal12-2}, Theorem \ref{t-bdc-cut-main} also holds when $t_n$ is replaced by the following constant
\[
 s_n=\frac{1}{\lambda_{n,1}}+\cdots+\frac{1}{\lambda_{n,n}},
\]
where $\lambda_{n,1},...,\lambda_{n,n}$ are nonzero eigenvalues of $I-K_n$. Furthermore, Theorem \ref{t-bdc-cut-main} also holds in separation with $\delta\in[1/2,1)$. We will use Theorem \ref{t-bdc-cut-main} to study the cutoff of several examples including the following theorem which concerns random walks with bottlenecks. It is a special case of Theorem \ref{t-btnk}.

\begin{thm}\label{t-btnk1}
For $n\ge 1$, let $\Omega_n=\{0,1,...,n\}$, $\pi_n\equiv 1/(n+1)$ and $K_n$ be an irreducible birth and death chain on $\Omega_n$ satisfying
\[
 K_n(i-1,i)=K_n(i,i-1)=\begin{cases}1/2&\text{for }i\notin\{x_{n,1},...,x_{n,k_n}\}\\\epsilon_n&\text{for }i=x_{n,j},\,1\le j\le k_n\end{cases},
\]
where $0\le k_n\le n$, $\epsilon_n\in(0,1/2]$, $x_{n,1},...,x_{n,k_n}\in\Omega_n$ are distinct and the holding rate at $i$ is adjusted accordingly. Set $t_n=n^2+a_n/\epsilon_n$, where
\[
 a_n=\sum_{i=1}^{k_n}\min\{x_{n,i},n+1-x_{n,i}\},
\]
and set
\[
 b_n=\max_{j:j\le n/2}\{(j+1)\times|\{1\le i\le k_n:j<x_{n,i}\le n-j\}|\}.
\]
Then, for any $\epsilon\in(0,1/2)$ and $\delta\in(0,1)$, there is $C=C(\epsilon,\delta)>1$ such that
\[
 C^{-1}t_n\le \min\{T_{n,\textnormal{\tiny TV}}^{(c)}(\epsilon),T_{n,\textnormal{\tiny TV}}^{(\delta)}(\epsilon)\}\le\max\{T_{n,\textnormal{\tiny TV}}^{(c)}(\epsilon),T_{n,\textnormal{\tiny TV}}^{(\delta)}(\epsilon)\}\le Ct_n,
\]
for $n$ large enough.

Moreover, the following are equivalent.
\begin{itemize}
\item[(1)] $\mathcal{F}_c$ has a total variation cutoff.

\item[(2)] For $\delta\in(0,1)$, $\mathcal{F}_\delta$ has a total variation cutoff.

\item[(3)] $a_n/(n^2\epsilon_n)\ra\infty$ and $a_n/b_n\ra\infty$.
\end{itemize}
\end{thm}

The remaining of this article is organized as follows. In Section 2, the concepts of cutoffs and mixing times and fundamental results are reviewed. In Section 3, we give a proof for Theorems \ref{t-mixingtime} and \ref{t-spectralgap}. For illustration, we consider several nontrivial examples in Section 4, where the mixing time and cutoff are determined. Note that the assumption regarding birth and death rates in Sections 3 and 4 can be relaxed using the comparison technique in \cite{DS93-1,DS93-2}.

\section{Backgrounds}

Throughout this paper, for any two sequences $s_n,t_n$ of positive numbers, we write $s_n=O(t_n)$ if there are $C>0,N>0$ such that $|s_n|\le C|t_n|$ for $n\ge N$. If $s_n=O(t_n)$ and $t_n=O(s_n)$, we write $s_n\asymp t_n$. If $t_n/s_n\ra 1$ as $n\ra\infty$, we write $t_n\sim s_n$.

\subsection{Cutoffs and mixing time}
Consider the following definitions.
\begin{defn}\label{d-cutoff}
Referring to the notation in (\ref{eq-tv}), a family $\mathcal{F}=\{(\Omega_n,K_n,\pi_n)|n=1,2,...\}$ is said to present a total variation
\begin{itemize}
\item[(1)] precutoff if there is a sequence $t_n$ and $B>A>0$ such that
\[
    \lim_{n\ra\infty}d_{n,\text{\tiny TV}}(\lceil Bt_n\rceil)=0,\quad\liminf_{n\ra\infty}d_{n,\text{\tiny TV}}(\lfloor At_n\rfloor)>0.
\]

\item[(2)] cutoff if there is a sequence $t_n$ such that, for all $\epsilon>0$,
\[
    \lim_{n\ra\infty}d_{n,\text{\tiny TV}}(\lceil(1+\epsilon)t_n\rceil)=0,\quad\lim_{n\ra\infty}d_{n,\text{\tiny TV}}(\lfloor(1-\epsilon)t_n\rfloor)=1.
\]
\end{itemize}
\end{defn}
In definition \ref{d-cutoff}(2), $t_n$ is called a cutoff time. The definition of a cutoff for continuous semigroups is similar with $\lceil\cdot\rceil$ and $\lfloor\cdot\rfloor$ deleted.

\begin{rem}
In Definition \ref{d-cutoff}, if $t_n\ra\infty$ (or equivalently $T_{n,\text{\tiny TV}}(\epsilon)\ra\infty$ for some $\epsilon\in(0,1)$), then the cutoff is consistent with (\ref{eq-cutmix}). This is also true for cutoffs in continuous semigroups without the assumption $t_n\ra\infty$. See \cite{GY-PhD,CSal08} for further discussions on cutoffs.
\end{rem}

It is well-known that the mixing time can be bounded below by the reciprocal of the spectral gap up to a multiple constant. We cite the bound in \cite{CSal12-2} as follows.

\begin{lem}\label{l-gapsing}
Let $K$ be an irreducible transition matrix on a finite set $\Omega$ with stationary distribution $\pi$. For $\delta\in(0,1)$, let $K_\delta$ be the $\delta$-lazy walk given by \textnormal{(\ref{eq-lazyK})}. Suppose $(\pi,K)$ is reversible, that is, $\pi(x)K(x,y)=\pi(y)K(y,x)$ for all $x,y\in\Omega$ and let $\lambda$ be the smallest non-zero eigenvalue of $I-K$. Then, for $\epsilon\in(0,1/2)$,
\[
 T_{\textnormal{\tiny TV}}^{(c)}(\epsilon)\ge\frac{-\log(2\epsilon)}{\lambda},\quad T_{\textnormal{\tiny TV}}^{(\delta)}(\epsilon)\ge\left\lfloor\frac{-\log(2\epsilon)}{2\max\{1-\delta,\log(2/\delta)\}\lambda}\right\rfloor,
\]
where the second inequality requires $|\Omega|\ge 2/\delta$.
\end{lem}

\subsection{Cutoffs for birth and death chains}

Consider a family of irreducible birth and death chains
\[
 \mathcal{F}=\{(\Omega_n,K_n,\pi_n)|n=1,2,...\},
\]
where $\Omega_n=\{0,1,...,n\}$ and $K_n$ has birth rate $p_{n,i}$, death rate $q_{n,i}$ and holding rate $r_{n,i}$. We write $\mathcal{F}_c,\mathcal{F}_\delta$ as families of the corresponding continuous time chains and $\delta$-lazy discrete time chains in $\mathcal{F}$. A criterion on total variation cutoffs for families of birth and death chains was introduced in \cite{DLP10}, which say that, for $\delta\in(0,1)$, $\mathcal{F}_c,\mathcal{F}_\delta$ have total variation cutoffs if and only if the product of the mixing time and the spectral gap tends to infinity. As the total variation distance is comparable with the separation distance, the authors of \cite{DLP10} identify cutoffs in total variation and separation, where a criterion on separation cutoffs was proposed in \cite{DS06}. In the recent work \cite{CSal12-2}, the cutoffs for $\mathcal{F}_c$ and $\mathcal{F}_\delta$ are proved to be equivalent and this leads to the following theorems.

\begin{thm}\cite[Section 4]{CSal12-2}\label{t-csal12-2}
Let $\mathcal{F}=\{(\Omega_n,K_n,\pi_n)|n=1,2,...\}$ be a family of irreducible birth and death chain with $\Omega_n=\{0,1,...,n\}$. For $n\ge 1$, let $\lambda_{n,1},...,\lambda_{n,n}$ be nonzero eigenvalues of $I-K_n$ and set
\[
 \lambda_n=\min_{1\le i\le n}\lambda_{n,i},\quad s_n=\frac{1}{\lambda_{n,1}}+\cdots+\frac{1}{\lambda_{n,n}}.
\]
Then, the following are equivalent.
\begin{itemize}
\item[(1)] $\mathcal{F}_c$ has a total variation cutoff.

\item[(2)] $\mathcal{F}_\delta$ has a total variation cutoff.

\item[(3)] $\mathcal{F}_c$ has a total variation precutoff.

\item[(4)] $\mathcal{F}_\delta$ has a total variation precutoff.

\item[(5)] $T_{n,\textnormal{\tiny TV}}^{(c)}(\epsilon)\lambda_n\ra\infty$ for some $\epsilon\in(0,1)$.

\item[(6)] $T_{n,\textnormal{\tiny TV}}^{(\delta)}(\epsilon)\lambda_n\ra\infty$ for some $\epsilon\in(0,1)$.

\item[(7)] $s_n\lambda_n\ra\infty$.
\end{itemize}
\end{thm}

\begin{thm}\cite[Section 4]{CSal12-2}\label{t-mixingcomp}
Referring to Theorem \ref{t-csal12-2}, it holds true that, for $\epsilon,\eta\in(0,1/2)$ and $\delta\in(0,1)$,
\[
 T_{n,\textnormal{\tiny TV}}^{(c)}(\epsilon)\asymp T_{n,\textnormal{\tiny TV}}^{(\delta)}(\eta).
\]
Further, if there is $\epsilon_0\in(0,1/2)$ such that $T_{n,\textnormal{\tiny TV}}^{(c)}(\epsilon_0)\lambda_n$ or $T_{n,\textnormal{\tiny TV}}^{(\delta)}(\epsilon_0)\lambda_n$ is bounded, then, for any $\epsilon\in(0,1/2)$ and $\delta\in(0,1)$,
\[
 T_{n,\textnormal{\tiny TV}}^{(c)}(\epsilon)\asymp T_{n,\textnormal{\tiny TV}}^{(\delta)}(\epsilon)\asymp \lambda_n^{-1}.
\]
\end{thm}

\subsection{A remark on the precutoff}

Note that if there is no cutoff in total variation, the approximation in Theorem \ref{t-mixingcomp} may fail for $\epsilon\in(1/2,1)$. This means that, for $0<\epsilon<1/2<\eta<1$, the orders of $T_{n,\text{\tiny TV}}^{(c)}(\epsilon)$ and $T_{n,\text{\tiny TV}}^{(c)}(\eta)$ can be different. Consider the following example. For $n\ge 3$, let $\Omega_n=\{0,1,...,n\}$, $M_n=\lfloor n/2\rfloor$ and
\begin{equation}\label{eq-precut}
 \begin{cases}K_n(i,i+1)=K_n(i+1,i)=1/2\quad\text{for }0\le i<n,\,i\ne M_n\\K_n(M_n,M_n+1)=K_n(M_n+1,M_n)=\epsilon_n\\K_n(0,0)=K_n(n,n)=1/2\\K_n(M_n,M_n)=K_n(M_n+1,M_n+1)=1/2-\epsilon_n\end{cases},
\end{equation}
with $\epsilon_n\le 1/2$. Assume that $\epsilon_n=o(n^{-2})$. By Theorem \ref{t-btnk1}, we have
\[
 T_{n,\text{\tiny TV}}^{(c)}(\epsilon)\asymp T_{n,\text{\tiny TV}}^{(\delta)}(\epsilon)\asymp n/\epsilon_n,\quad\forall \epsilon\in(0,1/2),\,\delta\in(0,1).
\]

Next, we consider the $\delta$-lazy discrete time case with $\delta=1/2$. Let $K_{n,1/2}=(I+K_n)/2$ and $K_n'$ be the $1/2$-lazy simple random walk on $\{0,1,...,M_n\}$, that is,
\[
 \begin{cases}K_n'(i,i+1)=K_n'(i+1,i)=1/4,\quad\forall 0\le i<M_n\\K_n'(i,i)=1/2,\quad\forall 0<i<M_n\\K_n'(0,0)=K_n'(M_n,M_n)=3/4\end{cases}.
\]
For $n\ge 3$, set
\[
 c_n=\min_{0\le i,j\le M_n}\frac{K_{n,1/2}^{m_n}(i,j)}{(K'_n)^{m_n}(i,j)},\quad C_n=\max_{0\le i,j\le M_n}\frac{K_{n,1/2}^{m_n}(i,j)}{(K'_n)^{m_n}(i,j)}.
\]

\begin{prop}
If $m_n\asymp n^2$, then
\[
 c_n\ra 1,\quad C_n\ra 1,\quad \text{as }n\ra\infty.
\]
\end{prop}
\begin{proof}
For $\ell\ge 1$, let $(i_0,i_1,...,i_\ell)$ be a path in $\{0,1,...,M_n\}$. Note that
\[
 \prod_{k=1}^\ell K_{n,1/2}(i_{k-1},i_k)\ge \left(\frac{3/4-\epsilon_n/2}{3/4}\right)^\ell\prod_{k=1}^\ell K_n'(i_{k-1},i_k).
\]
This implies $c_n\ge (1-2\epsilon_n/3)^{m_n}\sim 1$ as $n\ra\infty$. To see an upper bound of $C_n$, one may use Lemma 4.4 in \cite{DLP10} to conclude that, for $0\le i\le n$ and $\ell\ge 0$,
\[
 \begin{cases}K_{n,1/2}^\ell(i,j)\ge K_{n,1/2}^\ell(i,j-1)&\forall 1\le j\le i\\
 K_{n,1/2}^\ell(i,j)\ge K_{n,1/2}^\ell(i,j+1)&\forall i\le j<n\end{cases},
\]
and, for $0\le i\le M_n$ and $\ell\ge 0$,
\[
 \begin{cases}(K_n')^\ell(i,j)\ge (K_n')^\ell(i,j-1)&\forall 1\le j\le i\\
 (K_n')^\ell(i,j)\ge (K_n')^\ell(i,j+1)&\forall i\le j<M_n\end{cases}.
\]
By the induction, the above observation implies that, for any probabilities $\mu,\nu$ on $\{0,...,n\},\{0,...,M_n\}$ satisfying $\mu(i)=\nu(i)$ for $0\le i\le M_n$,
\[
 \mu K_{n,1/2}^\ell(j)\le \nu (K_n')^\ell(j),\quad\forall 0\le j\le M_n,\,\ell\ge 0.
\]
This yields $C_n\le 1$ for all $n\ge 3$.
\end{proof}

For $\epsilon\in(0,1)$, let $T_{n,\text{\tiny TV}}'(\epsilon)$ be the total variation mixing time for $K_n'$. It is well-known that, for $\epsilon\in(0,1)$, $T_{n,\text{\tiny TV}}'(\epsilon)\asymp n^2$. Let $d_{n,\text{\tiny TV}}^{(1/2)},d_{n,\text{\tiny TV}}'$ be the total variation distance for $K_{n,1/2},K_n'$. As a consequence of the above discussion, we obtain, for $\epsilon\in(0,1)$,
\[
 \limsup_{n\ra\infty}d_{n,\text{\tiny TV}}^{(1/2)}(T_{n,\text{\tiny TV}}'(\epsilon))\le\frac{1}{2}\left(1+\limsup_{n\ra\infty}d_{n,\text{\tiny TV}}'(T_{n,\text{\tiny TV}}'(\epsilon)\right)\le\frac{1+\epsilon}{2}.
\]
Thus, for $\epsilon\in(1/2,1)$, $T_{n,\text{\tiny TV}}^{(1/2)}(\epsilon)=O(n^2)$. Note that, for $m_n=o(n^2)$,
\[
 \lim_{n\ra\infty}\sum_{i\le an}K_{n,1/2}^{m_n}(0,i)=1,\quad\forall a>0.
\]
This yields $n^2=O(T_{n,\text{\tiny TV}}^{(1/2)}(\epsilon))$ for $\epsilon>0$. The above discussion is also valid for the continuous time case and any $\delta$-lazy discrete time case. We summarizes the results in the following theorem.

\begin{thm}
Let $\mathcal{F}=\{(\Omega_n,K_n,\pi_n)|n=1,2,...\}$ be the family of birth and death chains in \textnormal{(\ref{eq-precut})} and $\delta\in(0,1)$. Suppose that $\epsilon_n=o(n^{-2})$. Then, there is no total variation cutoff for $\mathcal{F}_c$ and $\mathcal{F}_\delta$. Furthermore, for $\epsilon\in(0,1/2)$,
\[
 T_{n,\textnormal{\tiny TV}}^{(c)}(\epsilon)\asymp T_{n,\textnormal{\tiny TV}}^{(\delta)}(\epsilon)\asymp n/\epsilon_n,
\]
and, for $\epsilon\in(1/2,1)$,
\[
 T_{n,\textnormal{\tiny TV}}^{(c)}(\epsilon)\asymp T_{n,\textnormal{\tiny TV}}^{(\delta)}(\epsilon)\asymp n^2.
\]
\end{thm}
\begin{rem}\label{r-precut}
Figure \ref{precut} displays the total variaton distances of the birth and death chains on $\{1,2,...,100\}$ with transition matrices $K_1$ and $K_2$ given by
\[
 \begin{cases}K_1(i,i)=1/2,&\text{for }i\notin\{1,50,51,100\}\\K_1(i,i+1)=K_1(i+1,i)=1/4,&\text{for $i<50$ or $i>51$}\\K_1(i,i)=3/4&\text{for }k\in\{1,100\}\\K_1(i,i+1)=K_1(i+1,i)=10^{-3}&\text{for }k=50\\K_1(i,i)=K_1(i,i)=3/4-10^{-3}&\text{for }i\in\{50,51\}\\K_1(i,j)=0&\text{otherwise}\end{cases}
\]
and
\[
 \begin{cases}K_2(i,i+1)=K_2(i+1,i)=10^{-2}&\text{for }i=25\\K_2(i,i)=3/4-10^{-2}&\text{for }i\in\{25,26\}\\K_2(i,j)=K_1(i,j)&\text{otherwise}
\end{cases}.
\]
Note that each curve has only one sharp transition for $d_{\text{\tiny TV}}(t)\le 1/2$. This is consistent with Theorem \ref{t-bdc-cut-main}. These examples show that multiple sharp transitions may occur for $d_{\text{\tiny TV}}(t)>1/2$. Note also that the flat part of the curves occupy very large time regions. For instance, the left most curve stays near the value $1/2$ for $t$ between $10^3$ and $10^6$.
\end{rem}

\begin{figure}[ht]
\includegraphics[width=6cm]{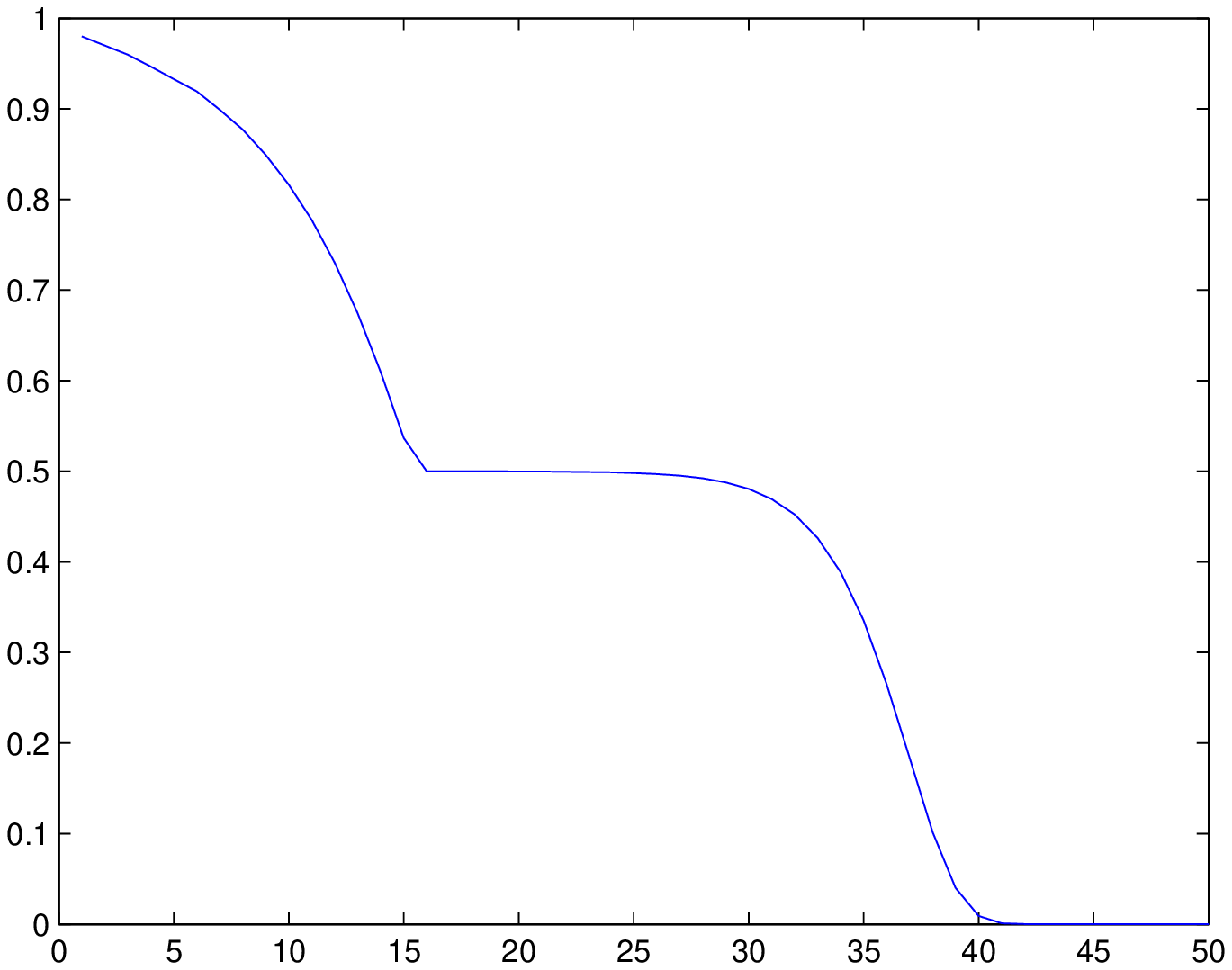}\includegraphics[width=6cm]{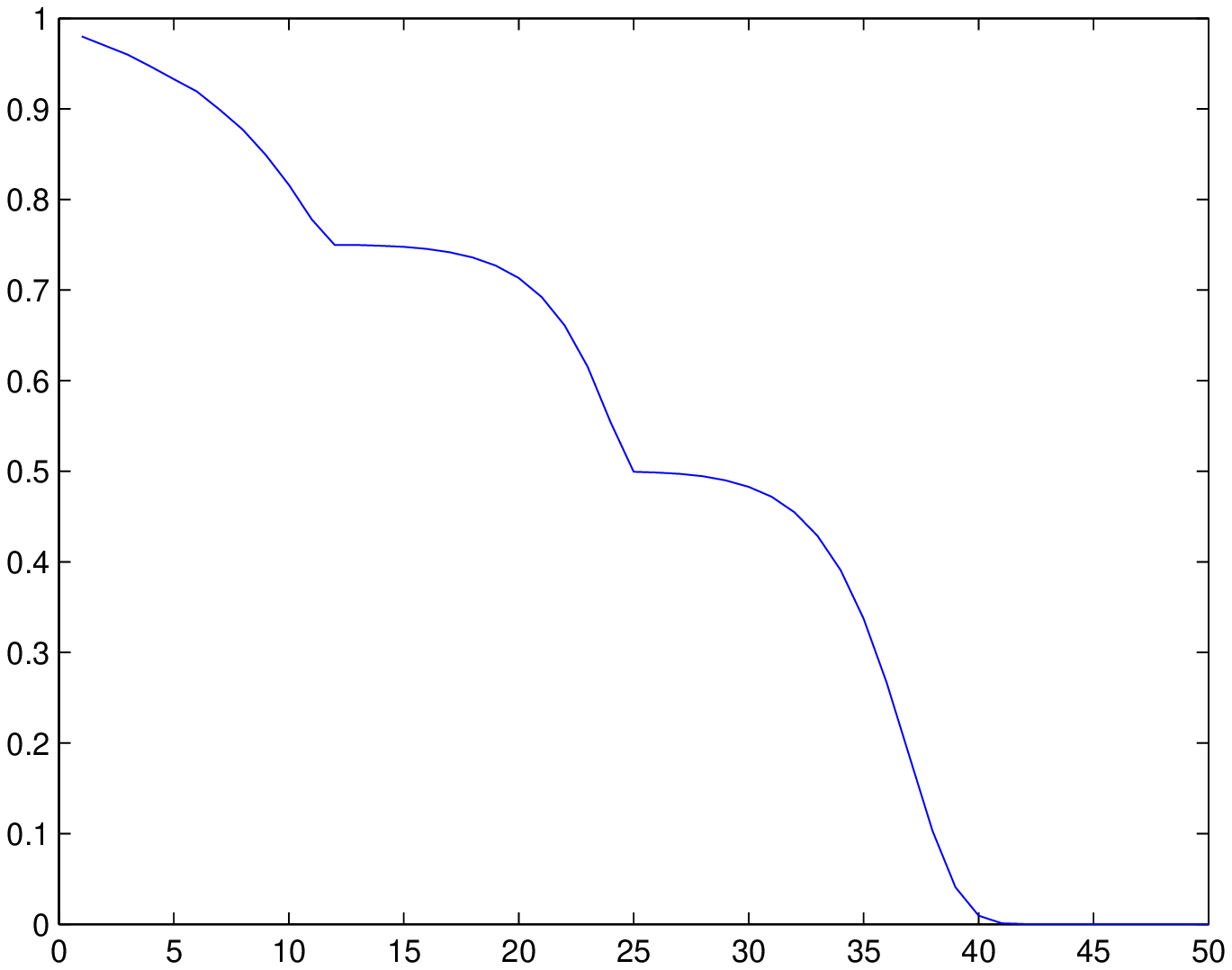}
\caption{The curves display the total variation distance of the chains in Remark \ref{r-precut}, where the left most curve is for $K_1$ and the right most curve is for $K_2$. The curve consists of the points $(m,d_\text{\tiny TV}(100^{\lfloor 0.1\times m\rfloor}))$ with $m=1,2,...,50$. The right most point of each curve corresponds to $d_{\text{\tiny TV}}(t)$ with $t=10^{10}$.}
\label{precut}
\end{figure}

\section{Bounds for mixing time and spectral gap}
This section is dedicated to proving Theorems \ref{t-mixingtime} and \ref{t-spectralgap}. In the first two subsections, we treat respectively the upper and lower bounds of the total variation mixing time. This leads to Theorem \ref{t-mixingtime}. In the third subsection, we provide a relaxation of the choice of $i_n$ in Theorem \ref{t-bdc-cut-main}. In the last subsection, we introduce a bound on the spectral gap which includes Theorem \ref{t-spectralgap}.

\subsection{An upper bound of the mixing time}
Let $(\Omega,K,\pi)$ be an irreducible birth and death chain, where $\Omega=\{0,1,...,n\}$ and $K$ has birth rate $p_i$, death rate $q_i$ and holding rate $r_i$. Let $(X_m)_{m=0}^\infty$ be a realization of the discrete time chain. Obviously, if $N_t$ is a Poisson process with parameter $1$ and independent of $(X_m)_{m=0}^\infty$, then $(X_{N_t})_{t\ge 0}$ is a realization of the continuous time chain. For $\delta\in[0,1)$, if $(B_m^{(\delta)})_{m=1}^\infty$ is a sequence of independent Bernoulli($1-\delta$) trials which are independent of $(X_m)_{m=0}^\infty$, then $Y_m^{(\delta)}=X_{B_1^{(\delta)}+\cdots+B_m^{(\delta)}}$ is a realization of the $\delta$-lazy chain. For $0\le i\le n$, we define the first passage time to $i$ by
\begin{equation}\label{eq-passage}
 \widetilde{\tau}_i:=\inf\{t\ge 0|X_{N_t}=i\},\quad\tau_i^{(\delta)}:=\min\{m\ge 0|Y_m=i\},
\end{equation}
and simply put $\tau_i:=\tau_i^{(0)}=\min\{m\ge 0|X_m=i\}$. Briefly, we write $\mathbb{P}_i(\cdot)$ for $\mathbb{P}(\cdot|X_0=i)$ and write $\mathbb{E}_i,\text{Var}_i$ as the expectation and variance under $\mathbb{P}_i$. The main result of this subsection is as follows.

\begin{thm}[Upper bound]\label{t-upper}
Let $(\Omega,K,\pi)$ be an irreducible birth and death chain with $\Omega=\{0,1,...,n\}$. Let $\tau_i=\tau_i^{(0)}$ be the first passage time to $i$ defined in \textnormal{(\ref{eq-passage})}. For $\epsilon\in(0,1)$ and $\delta\in[1/2,1)$,
\begin{equation}\label{eq-uppermix}
 \max\left\{T_{\textnormal{\tiny TV}}^{(c)}(\epsilon),(1-\delta)T_{\textnormal{\tiny TV}}^{(\delta)}(\epsilon)\right\}\le\frac{9(\mathbb{E}_0\tau_{i_0}+\mathbb{E}_n\tau_{i_0})}{\epsilon^2},
\end{equation}
where $i_0\in\{0,...,n\}$ satisfies $\pi([0,i_0-1])\le 1/2$ and $\pi([i_0+1,n])\le 1/2$.
\end{thm}

\begin{rem}\label{r-upper}
The authors of \cite{CSal12-2} obtain a slightly improved upper bound similar to (\ref{eq-uppermix}), which says that
\[
 \max\left\{T_{\textnormal{\tiny TV}}^{(c)}(\epsilon),(1-\delta)T_{\textnormal{\tiny TV}}^{(\delta)}(\epsilon)\right\}\le\frac{(\sqrt{\epsilon}+\sqrt{1-\epsilon})
 (\mathbb{E}_0\tau_{i_0}+\mathbb{E}_n\tau_{i_0})}{\sqrt{\epsilon}}.
\]
Comparing with (\ref{eq-uppermix}), the above inequality has an improved dependence on $\epsilon$.
\end{rem}

To understand the right side of (\ref{eq-uppermix}), we introduce the following lemma.

\begin{lem}\label{l-passage}
Referring to the setting in \textnormal{(\ref{eq-passage})}, it holds true that, for $i<j$, $\mathbb{E}_i(\tau_j^{(\delta)})=\mathbb{E}_i(\tau_j)/(1-\delta)$ and
$\mathbb{E}_i(\tau_j)=\mathbb{E}_i(\widetilde{\tau}_j)=\sum_{k=i}^{j-1}\pi([0,k])/(p_k\pi(k))$.
\end{lem}
\begin{proof}
The proof is based on the strong Markov property. See \cite[Proposition 2]{BBF09} for a reference on the discrete time case, whereas the continuous time case is an immediate result of the fact $\{\widetilde{\tau}_i>t\}=\{\tau_i>N_t\}$.
\end{proof}

\begin{rem}
By Theorem \ref{t-upper} and Lemma \ref{l-passage}, the total variation mixing time for the continuous time and the $\delta$-lazy, with $\delta\ge 1/2$, discrete time birth and death chain on $\{0,1,...,n\}$ are bounded above by the following term up to a multiple constant.
\[
 \sum_{k=0}^{i_0-1}\frac{\pi([0,k])}{p_k\pi(k)}+\sum_{k=i_0+1}^n\frac{\pi([k,n])}{q_k\pi(k)},
\]
where $i_0\in\{0,...,n\}$ satisfies $\pi([0,i_0-1])\le 1/2$ and $\pi([i_0+1,n])\le 1/2$.
\end{rem}

\begin{rem}
In Theorem \ref{t-upper}, $i_0$ is unique if $\pi([0,i])\ne 1/2$ for all $0\le i\le n$. If $\pi([0,j])=1/2$, then $i_0$ can be $j$ or $j+1$, but the right side of (\ref{eq-uppermix}) is the same in either case using Lemma \ref{l-passage}.
\end{rem}

\begin{rem}\label{r-gap}
Let $K$ be an irreducible birth and death chain with birth, death and holding rates $p_i,q_i,r_i$ and stationary distribution $\pi$. Let $\lambda$ be the spectral gap of $K$. As a consequence of Lemma \ref{l-gapsing} and theorem \ref{t-upper}, we obtain, for $\epsilon\in(0,1/2)$,
\[
 \lambda\ge\frac{\epsilon^2\log(1/(2\epsilon))}{9}\left(\sum_{k=0}^{i_0-1}\frac{\pi([0,k])}{p_k\pi(k)}+\sum_{k=i_0+1}^n\frac{\pi([k,n])}{q_k\pi(k)}\right)^{-1},
\]
where $i_0$ is such that $\pi([0,i_0-1])\le 1/2$ and $\pi([i_0+1,n])\le 1/2$. The maximum of $\epsilon^2\log(1/(2\epsilon))$ on $(0,1/2)$ is attained at $\epsilon=1/(2\sqrt{e})$ and equal to $1/(8e)$. A similar lower bound of the spectral gap is also derived in \cite{CSal12} with improved constant.
\end{rem}

As a simple application of Lemma \ref{l-passage}, we have
\begin{cor}\label{c-tv}
Referring to Lemma \ref{l-passage}, for $i\le j$,
\[
 \mathbb{E}_i\tau_j\le \left(\frac{1}{\pi([j,n])}-1\right)\mathbb{E}_n\tau_i.
\]
\end{cor}
\begin{proof}
By Lemma \ref{l-passage}, one has
\[
 \mathbb{E}_i\tau_j=\sum_{k=i}^{j-1}\frac{\pi([0,k])}{p_k\pi(k)},\quad\mathbb{E}_n(\tau_i)=\sum_{k=i}^{n-1}\frac{\pi([k+1,n])}{q_{k+1}\pi(k+1)}=\sum_{k=i}^{n-1}\frac{\pi([k+1,n])}{p_{k}\pi(k)}.
\]
The inequality is then given by the fact $\pi([0,k])/\pi([k+1,n])=1/\pi([k+1,n])-1\le 1/\pi([j,n])-1$ for $k<j$.
\end{proof}

The following proposition is the main technique used to prove Theorem \ref{t-upper}.
\begin{prop}\label{p-upper}
Referring to the setting in \textnormal{(\ref{eq-passage})}, it holds true that, for $j<k$,
\[
 d_{\textnormal{\tiny TV}}^{(c)}(i,t)\le\mathbb{P}_i(\max\{\widetilde{\tau}_j,\widetilde{\tau}_k\}>t)+1-\pi([j,k]),
\]
and
\[
 d_{\textnormal{\tiny TV}}^{(1/2)}(i,t)\le\mathbb{P}_i(\max\{\tau_j^{(1/2)},\tau_k^{(1/2)}\}>t)+1-\pi([j,k]),
\]
In particular,
\[
 d_{\textnormal{\tiny TV}}^{(c)}(t)\le\frac{\mathbb{E}_0\widetilde{\tau}_k+\mathbb{E}_n\widetilde{\tau}_j}{t}+1-\pi([j,k])
\]
and
\[
 d_{\textnormal{\tiny TV}}^{(1/2)}(t)\le\frac{2(\mathbb{E}_0\tau^{(1/2)}_k+\mathbb{E}_n\tau^{(1/2)}_j)}{t}+1-\pi([j,k]).
\]
\end{prop}

In the above proposition, the discrete time case is discussed in Lemma 2.3 in \cite{DLP10}. Our method to prove this proposition is to construct a no-crossing coupling. We give the proof of the continuous time case for completeness and refer to \cite{DLP10} for the discrete time case, where a heuristic idea on the construction of no-crossing coupling is proposed.

\begin{proof}[Proof of Proposition \ref{p-upper}]
Let $(Y_t)_{t\ge 0}$ be another process corresponding to $H_t$ with $Y_0\overset{d}{=}\pi$. Set $T:=\inf\{t\ge 0|X_t=Y_t\}$ and $Z_t:=Y_t\mathbf{1}_{\{t\le T\}}+X_t\mathbf{1}_{\{t>T\}}$. Clearly, $(X_t,Z_t)_{t\ge 0}$ is a coupling for the semigroup $H_t$ and must be no-crossing according to the continuous time setting. Note that $T=\inf\{t\ge 0|X_t=Z_t\}$ is the coupling time of $X_t$ and $Z_t$. The classical coupling statement implies that
\begin{equation}\label{eq-cts-tv}
 d_{\textnormal{\tiny TV}}^{(c)}(i,t)\le\mathbb{P}_i(T>t).
\end{equation}
See e.g. \cite{A83} for a reference. Note that $X_{\tau_j}=j$, $X_{\tau_k}=k$ and
\[
 \mathbb{P}_i(X_{\widetilde{\tau}_j}\le Y_{\widetilde{\tau}_j})=\pi([j,n]),\quad\mathbb{P}_i(X_{\widetilde{\tau}_k}\ge Y_{\widetilde{\tau}_k})=\pi([0,k]).
\]
As $X_t,Y_t$ can not cross each other without coalescing in advance, this implies
\begin{align}
 \mathbb{P}_i(T\le\max\{\widetilde{\tau}_j,\widetilde{\tau}_k\})
 &\ge\mathbb{P}_i(\min\{\widetilde{\tau}_j,\widetilde{\tau}_k\}\le T\le\max\{\widetilde{\tau}_j,\widetilde{\tau}_k\})\notag\\
 &\ge\mathbb{P}_i(X_{\widetilde{\tau}_j}\le Y_{\widetilde{\tau}_j},X_{\widetilde{\tau}_k}\ge Y_{\widetilde{\tau}_k})\ge\pi([j,k]).\notag
\end{align}
Putting this back to (\ref{eq-cts-tv}) gives the desired result.

For the last part, note that if $i\le j$, then $\widetilde{\tau}_j<\widetilde{\tau}_k$ and, by Markov's inequality, this implies
\[
 \mathbb{P}_i(\max\{\widetilde{\tau}_j,\widetilde{\tau}_k\}>t)\le\mathbb{P}_0(\widetilde{\tau}_k>t)
 \le\mathbb{E}_0\widetilde{\tau}_k/t.
\]
Similarly, for $i\ge k$, one can show that
\[
 \mathbb{P}_i(\max\{\widetilde{\tau}_j,\widetilde{\tau}_k\}>t)\le\mathbb{P}_n(\widetilde{\tau}_j>t)
 \le\mathbb{E}_n\widetilde{\tau}_j/t.
\]
For $j<i<k$, we have
\[
 \mathbb{P}_i(\max\{\widetilde{\tau}_j,\widetilde{\tau}_k\}>t)\le\mathbb{P}_i(\widetilde{\tau}_j>t)
 +\mathbb{P}_i(\widetilde{\tau}_k>t)\le\frac{\mathbb{E}_n\widetilde{\tau}_j+\mathbb{E}_0\widetilde{\tau}_k}{t}.
\]
\end{proof}

\begin{proof}[Proof of Theorem \ref{t-upper}]
Set $j_{\epsilon}=\min\{i\ge 0|\pi([0,i])\ge \epsilon/3\}$ and $k_\epsilon=\min\{i\ge 0|\pi([0,i])\ge 1-\epsilon/3\}$. By Proposition \ref{p-upper} and Lemma \ref{l-passage}, the choice of $j=j_\epsilon$ and $k=k_\epsilon$ implies that
\[
 T_{\text{\tiny TV}}^{(c)}(\epsilon)\le\frac{3(\mathbb{E}_0\tau_{k_\epsilon}+\mathbb{E}_n\tau_{j_\epsilon})}{\epsilon}.
\]
By Corollary \ref{c-tv}, one has
\[
 \mathbb{E}_0\tau_{k_\epsilon}=\mathbb{E}_0\tau_{i_0}+\mathbb{E}_{i_0}\tau_{k_\epsilon}\le\mathbb{E}_0\tau_{i_0}+\left(\frac{3}{\epsilon}-1\right)\mathbb{E}_n\tau_{i_0}
\]
and
\[
 \mathbb{E}_n\tau_{j_\epsilon}=\mathbb{E}_n\tau_{i_0}+\mathbb{E}_{i_0}\tau_{j_\epsilon}\le\mathbb{E}_n\tau_{i_0}+\left(\frac{3}{\epsilon}-1\right)\mathbb{E}_0\tau_{i_0}.
\]
Adding up both terms gives the upper bound in continuous time case. The proof for the $(1/2)$-lazy discrete time case is similar and, by Proposition \ref{p-upper}, we obtain $T_{\text{\tiny TV}}^{(1/2)}(\epsilon)\le 18(\mathbb{E}_0\tau_{i_0}+\mathbb{E}_n\tau_{i_0})/\epsilon^2$. For $\delta\in(1/2,1)$, note that $K_\delta=(K_{2\delta-1})_{1/2}$. Since the birth and death rates of $K_{2\delta-1}$ are $2(1-\delta)p_i$ and $2(1-\delta)q_i$, the above result and Lemma \ref{l-passage} lead to $T_{\text{\tiny TV}}^{(\delta)}(\epsilon)\le 9(\mathbb{E}_0\tau_{i_0}+\mathbb{E}_n\tau_{i_0})/((1-\delta)\epsilon^2)$.
\end{proof}

\subsection{A lower bound of the mixing time}
The goal of this subsection is to establish a lower bound on the total variation mixing time for birth and death chains. Recall the notations in the previous subsection. Let $(X_m)_{m=0}^\infty$ be an irreducible birth and death chain with transition matrix $K$ and stationary distribution $\pi$. Let $N_t$ be a Poisson process of parameter $1$ that is independent of $X_m$. For $0\le i\le n$, let $\tau_i=\min\{m\ge 0|X_m=i\}$ and $\widetilde{\tau}_i=\inf\{t\ge 0|X_{N_t}=i\}$. Then, the total variation mixing time satisfies
\begin{equation}\label{eq-lowerdisc}
 d_{\text{\tiny TV}}(0,t)\ge K^t(0,[0,i-1])-\pi([0,i-1])\ge\mathbb{P}_0(\tau_i>t)-\pi([0,i-1])
\end{equation}
and
\begin{equation}\label{eq-lowercts}
 d_{\text{\tiny TV}}^{(c)}(0,t)\ge H_t(0,[0,i-1])-\pi([0,i-1])\ge\mathbb{P}_0(\widetilde{\tau}_i>t)-\pi([0,i-1]).
\end{equation}
Brown and Shao discuss the distribution of $\widetilde{\tau}_i$ in \cite{BS87}, of which proof also works for the discrete time case. In detail, if $-1<\beta_1<\cdots<\beta_i<1$ are the eigenvalues of the submatrix of $K$ indexed by $\{0,...,i-1\}$ and $\lambda_j=1-\beta_j$, then
\begin{equation}\label{eq-disc}
 \mathbb{P}_0(\tau_i>t)=\sum_{j=1}^i\left(\prod_{k\ne j}\frac{\lambda_k}{\lambda_k-\lambda_j}\right)(1-\lambda_j)^t
\end{equation}
and
\begin{equation}\label{eq-cts}
 \mathbb{P}_0(\widetilde{\tau}_i>t)=\sum_{j=1}^i\left(\prod_{k\ne j}\frac{\lambda_k}{\lambda_k-\lambda_j}\right)e^{-t\lambda_j}.
\end{equation}
Note that, under $\mathbb{P}_0$, $\widetilde{\tau}_i$ is the sum of independent exponential random variables with parameters $\lambda_1,...,\lambda_i$. If $\beta_1>0$, then $\tau$ is the sum of independent geometric random variables with parameters $\lambda_1,...,\lambda_i$. In discrete time case, the requirement $\beta_1>0$ holds automatically for the $\delta$-lazy chain with $\delta\ge 1/2$. The above formula leads to the following theorem.

\begin{thm}[Lower bound]\label{t-lower}
Let $K$ be the transition matrix of an irreducible birth and death chain on $\{0,1,...,n\}$. Let $\tau_i=\tau_i^{(0)}$ be the first passage time to $i$ defined in \textnormal{(\ref{eq-passage})}. For $\delta\in[1/2,1)$,
\[
 \min\{T_{\textnormal{\tiny TV}}^{(c)}(1/10),2(1-\delta)T_{\textnormal{\tiny TV}}^{(\delta)}(1/20)\}\ge \frac{\max\{\mathbb{E}_0\tau_{i_0},\mathbb{E}_n\tau_{i_0}\}}{6},
\]
where $i_0\in\{0,...,n\}$ satisfies $\pi([0,i_0-1])\le 1/2$ and $\pi([i_0+1,n])\le 1/2$.
\end{thm}
\begin{proof}[Proof of Theorem \ref{t-lower}]
First, we consider the continuous time case. Let $\lambda_1,...,\lambda_i$ be eigenvalues of the submatrix of $I-K$ indexed by $0,...,i-1$ and $\widetilde{\tau}_{i,1},...,\widetilde{\tau}_{i,i}$ be independent exponential random variables with parameters $\lambda_1,...,\lambda_i$. By (\ref{eq-cts}), $\widetilde{\tau}_i$ and $\widetilde{\tau}_{i,1}+\cdots+\widetilde{\tau}_{i,i}$ are identically distributed under $\mathbb{P}_0$ and, by (\ref{eq-lowercts}), this implies
\[
 d_{\text{\tiny TV}}^{(c)}(0,t)\ge \mathbb{P}(\widetilde{\tau}_{i,1}+\cdots+\widetilde{\tau}_{i,i}>t)-\pi([0,i-1]).
\]
It is easy to see that
\[
 \mathbb{E}_0\widetilde{\tau}_i=\frac{1}{\lambda_1}+\cdots+\frac{1}{\lambda_i},\quad\text{Var}_0(\widetilde{\tau}_i)=\frac{1}{\lambda_1^2}+\cdots+\frac{1}{\lambda_i^2}.
\]
Let $a\in(0,1)$ and consider the following two cases. If $1/\lambda_j>a\mathbb{E}_0\widetilde{\tau}_i$ for some $1\le j\le i$, then
\[
 \mathbb{P}_0(\widetilde{\tau}_i>t)\ge\mathbb{P}(\widetilde{\tau}_{i,j}>t)>e^{-t/(a\mathbb{E}_0\widetilde{\tau}_i)}.
\]
If $1/\lambda_j\le a\mathbb{E}_0\widetilde{\tau}_i$ for all $1\le j\le i$, then $\text{Var}_0(\widetilde{\tau}_i)\le a(\mathbb{E}_0\widetilde{\tau}_i)^2$ and, by the one-sided Chebyshev inequality, we have
\[
 \mathbb{P}_0(\widetilde{\tau}_i>t)\ge \frac{(t-\mathbb{E}_0\widetilde{\tau}_i)^2}{\text{Var}_0(\widetilde{\tau}_i)+(t-\mathbb{E}_0\widetilde{\tau}_i)^2}\ge\frac{(t-\mathbb{E}_0\widetilde{\tau}_i)^2}{a(\mathbb{E}_0\widetilde{\tau}_i)^2+(t-\mathbb{E}_0\widetilde{\tau}_i)^2}=\frac{(1-b)^2}{a+(1-b)^2},
\]
for $t=b\mathbb{E}_0\widetilde{\tau}_i$ with $b\in(0,1)$. Combining both cases and setting $i=i_0$ in (\ref{eq-lowercts}) yields that, for $a,b\in(0,1)$,
\begin{equation}\label{eq-a/b}
 d_{\text{\tiny TV}}^{(c)}(0,b\mathbb{E}_0\widetilde{\tau}_{i_0})\ge\min\left\{e^{-b/a},\frac{(1-b)^2}{a+(1-b)^2}\right\}-\frac{1}{2}.
\end{equation}
Putting $a=1/3$ and $b=1/6$ gives $T_{\text{\tiny TV}}^{(c)}(0,1/10)\ge\mathbb{E}_0\widetilde{\tau}_{i_0}/6$.

For the discrete time case, note that the eigenvalues of the submatrix of $I-K_{1/2}=\frac{1}{2}(I-K)$ indexed by $0,...,i-1$ are $\lambda_1/2,...,\lambda_i/2$. Let $\tau_{i,1},...,\tau_{i,i}$ be independent geometric random variables with success probabilities $\lambda_1/2,...,\lambda_i/2$. Replacing $K$ with $K_{1/2}$ in (\ref{eq-lowerdisc}), we obtain
\[
 d_{\text{\tiny TV}}^{(1/2)}(0,t)\ge\mathbb{P}_0(\tau_{i,1}+\cdots+\tau_{i,i}>t)-\pi([0,i-1]).
\]
Note that, under $\mathbb{P}_0$, $\tau_i^{(1/2)}$ has the same distribution as $\tau_{i,1}+\cdots+\tau_{i,i}$ and this implies
\[
 \mathbb{E}_0\tau_i^{(1/2)}=\frac{2}{\lambda_1}+\cdots+\frac{2}{\lambda_i},\quad\text{Var}_0(\tau_i^{(1/2)})=\sum_{j=1}^i\frac{4(1-\lambda_j/2)}{\lambda_j^2}\le\sum_{j=1}^i\frac{4}{\lambda_j^2}.
\]
Using the same analysis as before, one may derive, for $1/\mathbb{E}_0\tau_i^{(1/2)}<a<1$ and $t<\mathbb{E}_0\tau_i^{(1/2)}$,
\[
 \mathbb{P}_0(\tau_i^{(1/2)}>t)\ge\min\left\{\left(1-\frac{1}{a\mathbb{E}_0\tau_i^{(1/2)}}\right)^t,\frac{\left(t-\mathbb{E}_0\tau_i^{(1/2)}\right)^2}{a\left(\mathbb{E}_0\tau_i^{(1/2)}\right)^2+\left(t-\mathbb{E}_0\tau_i^{(1/2)}\right)^2}\right\}.
\]
By Lemma \ref{l-passage}, $\mathbb{E}_0\tau_i^{(1/2)}\ge 2i$. Obviously, if $i_0=0$, then $T_{\text{\tiny TV}}^{(1/2)}(0,1/20)\ge 0=\mathbb{E}_0\tau_{i_0}^{(1/2)}$. For $i_0\ge 1$, $\mathbb{E}_0\tau_{i_0}^{(1/2)}\ge 2$ and the setting, $a=2/3$ and $t=\left\lfloor\mathbb{E}_0\tau_{i_0}^{(1/2)}/12\right\rfloor$, implies
\[
 d_{\text{\tiny TV}}^{(1/2)}\left(0,\left\lfloor\mathbb{E}_0\tau_{i_0}^{(1/2)}/12\right\rfloor\right)\ge\min\left\{2^{-1/3},\frac{(11/12)^2}{2/3+(11/12)^2}\right\}-\frac{1}{2}>\frac{1}{20},
\]
where the first inequality use the fact that $s\log(1-3/(2s))$ is increasing on $[2,\infty)$. Hence, we have $T_{\text{\tiny TV}}^{(1/2)}(0,1/20)\ge\mathbb{E}_0\tau_{i_0}^{(1/2)}/12=\mathbb{E}_0\tau_{i_0}/6$. For $\delta>1/2$, the combination of the above result and the observation $K_\delta=(K_{2\delta-1})_{1/2}$ implies that $T_{\text{\tiny TV}}^{(\delta)}(0,1/20)\ge\mathbb{E}_0\tau_{i_0}/(12(1-\delta))$.

The analysis from the other end point gives the other lower bound. This finishes the proof.
\end{proof}

\subsection{Relaxation of the median condition}

In some cases, it is not easy to determine the value of $i_n$ in Theorem \ref{t-bdc-cut-main}. Let $t_n$ be the constants in Theorem \ref{t-upper}.
For $c\in(0,1)$, let $i_n(c)\in\{0,...,n\}$ be the state such that $\pi_n([0,i_n(c)-1])\le c$, $\pi_n([i_n(c)+1,n])\le 1-c$ and let $t_n(c)$ be the following constant
\[
 t_n(c)=\sum_{k=0}^{i_n(c)-1}\frac{\pi_n([0,k])}{\pi_n(k)p_{n,k}}+\sum_{k=i_n(c)+1}^n\frac{\pi_n([k,n])}{\pi_n(k)q_{n,k}}.
\]
Assume that $c\ge 1/2$. In this case, if $i_n$ is the smallest median, then $i_n\le i_n(c)$ and
\[
 \sum_{k=i_n}^{i_n(c)-1}\frac{\pi([0,k])}{\pi_n(k)p_{n,k}}=\sum_{k=i_n+1}^{i_n(c)}\frac{\pi_n([0,k-1])}{\pi_n(k)q_{n,k}}.
\]
Note that, for $i_n<k\le i_n(c)$,
\[
 \frac{1}{2}\le\pi_n([0,i_n])\le\frac{\pi_n([0,k-1])}{\pi_n([k,n])}\le\frac{1}{\pi_n([i_n(c),n])}\le\frac{1}{1-c}.
\]
This implies $t_n/2\le t_n(c)\le t_n/(1-c)$. Similarly, for $c\le 1/2$, one can show that $t_n/2\le t_n(c)\le t_n/c$. Combining both cases gives
\begin{equation}\label{eq-compt_n}
 t_n/2\le t_n(c)\le t_n/\min\{c,1-c\}.
\end{equation}
As a consequence of the above discussion, we obtain the following theorem.

\begin{thm}\label{t-tvmixingcn}
Referring to Theorem \ref{t-bdc-cut-main}. For $n\ge 1$, let $j_n\in\{0,1,...,n\}$ and set
\[
 t_n'=\max\left\{\sum_{k=0}^{j_n-1}\frac{\pi_n([0,k])}{\pi_n(k)p_{n,k}},\sum_{k=j_n+1}^n\frac{\pi_n([k,n])}{\pi_n(k)q_{n,k}}\right\}.
\]
Suppose that
\[
 0<\liminf_{n\ra\infty}\pi_n([0,j_n])\le\limsup_{n\ra\infty}\pi_n([0,j_n])<1.
\]
Then, Theorem \ref{t-bdc-cut-main} remains true if $t_n$ is replaced by $t_n'$.
\end{thm}

\begin{proof}
The proof comes immediately from (\ref{eq-compt_n}) with $c=\pi_n([0,j_n])$.
\end{proof}

We use this observation to bound the cutoff time in the following theorem.

\begin{thm}\label{t-tvmixingcts}
Referring to Theorem \ref{t-bdc-cut-main}. Suppose that $\mathcal{F}_c$ has a total variation cutoff. Then, for any $\epsilon\in(0,1)$,
\[
 \frac{2\log 2}{5}\le\liminf_{n\ra\infty}\frac{T_{n,\textnormal{\tiny TV}}^{(c)}(\epsilon)}{t_n}\le\limsup_{n\ra\infty}\frac{T_{n,\textnormal{\tiny TV}}^{(c)}(\epsilon)}{t_n}\le 2
\]
\end{thm}
\begin{proof}[Proof of Theorem \ref{t-tvmixingcts}]
The upper bound is given by Remark \ref{r-upper} and the fact, $\max\{s,t\}\ge (s+t)/2$, whereas the lower bound is obtained by applying $a=2/5$ and $b=a\log(2/(1+2\epsilon))$ in (\ref{eq-a/b}) with $\epsilon\ra 0$.
\end{proof}

\subsection{Bounding the spectral gap}
This subsection is devoted to poviding bounds on the specral gap for birth and death chains. As the graph associated with a birth and death chain is a path, weighted Hardy's inequality can be used to bound the spectral gap. We refer to the Appendix for a detailed discussion of the following results.
See Theorems \ref{t-Hardy}-\ref{t-sym1}.

\begin{thm}\label{t-gap}
Consider an irreducible birth and death chain on $\{0,...,n\}$ with birth, death and holding rates $p_i,q_i,r_i$ and stationary distribution $\pi$. Let $\lambda$ be the spectral gap and set, for $0\le i\le n$,
\[
 C(i)=\max\left\{\max_{j:j<i}\sum_{k=j}^{i-1}\frac{\pi([0,j])}{\pi(k)p_k},\max_{j:j>i}\sum_{k=i+1}^j\frac{\pi([j,n])}{\pi(k)q_k}\right\}.
\]
Then, for $0\le m\le n$,
\[
 \frac{1}{4C(m)}\le\lambda\le\frac{1}{\min\{\pi([0,m]),\pi([m,n])\}C(m)}.
\]
In particular, if $M$ is a median of $\pi$, that is, $\pi([0,M])\ge 1/2$ and $\pi([M,n])\ge 1/2$, then
\[
 \frac{1}{4C(M)}\le\lambda\le\frac{2}{C(M)}.
\]
\end{thm}

\begin{thm}\label{t-symgap}
Consider an irreducible birth and death chain on $\{0,...,n\}$ with birth, death and holding rates $p_i,q_i,r_i$ and stationary distribution $\pi$. Let $\lambda$ be the spectral gap and set $N=\lceil n/2\rceil$. Suppose that $p_i=q_{n-i}$ for $0\le i\le n$. Then,
\[
 \frac{1}{4C}\le\lambda\le\frac{1}{C},
\]
where
\[
 C=\max_{0\le i\le N-1}\left\{\pi([0,i])\sum_{j=i}^{N-1}\frac{1}{\pi(j)p_j}\right\}\quad\text{if $n$ is even},
\]
and
\[
 C=\max_{0\le i\le N-1}\left\{\pi([0,i])\left(\sum_{j=i}^{N-2}\frac{1}{\pi(j)p_j}+\frac{1}{2\pi(N-1)p_{N-1}}\right)\right\}\quad\text{if $n$ is odd}.
\]
\end{thm}

\begin{rem}
In \cite{SC99}, the author also obtained bounds similar to Theorem \ref{t-symgap} for the case $\pi(i)\ge\pi(i+1)$ with $0\le i<n/2$ using the path technique. For more information on path techniques, see \cite{DS93-1,DS93-2,DS91} and the references therein.
\end{rem}

\section{Examples}
In this section, we will apply the theory developed in the previous section to examples of special interest. First, we give a criterion on the cutoff using the birth and death rates.

\begin{thm}[Cutoffs from birth and death rates]\label{t-cutoffrates}
Let $\mathcal{F}=\{(\Omega_n,K_n,\pi_n)|n=1,2,...\}$ be a family of irreducible birth and death chains on $\Omega_n=\{0,1,...,n\}$ with birth rate, $p_{n,i}$, death rate $q_{n,i}$ and holding rate $r_{n,i}$. Let $\lambda_n$ be the spectral gap of $K_n$. For $n\ge 1$, let $j_n\in\{0,...,n\}$ and set
\[
 t_n=\max\left\{\sum_{k=0}^{j_n-1}\frac{\pi_n([0,k])}{\pi_n(k)p_{n,k}},\sum_{k=j_n+1}^n\frac{\pi_n([k,n])}{\pi_n(k)q_{n,k}}\right\}
\]
and
\[
 \ell_n=\max\left\{\max_{j:j<j_n}\sum_{k=j}^{j_n-1}\frac{\pi_n([0,j])}{\pi_n(k)p_{n,k}},\max_{j:j>j_n}\sum_{k=j_n+1}^j\frac{\pi_n([j,n])}{\pi_n(k)q_{n,k}}\right\}.
\]
Suppose that
\[
 0<\liminf_{n\ra\infty}\pi_n([0,j_n])\le\limsup_{n\ra\infty}\pi_n([0,j_n])<1.
\]
Then, for $\epsilon\in(0,1/2)$ and $\delta\in(0,1)$,
\[
 \lambda_{n}\asymp 1/\ell_n, \quad T_{n,\textnormal{\tiny TV}}^{(c)}(\epsilon)\asymp t_n\asymp T_{n,\textnormal{\tiny TV}}^{(\delta)}(\epsilon).
\]
Furthermore, the following are equivalent.
\begin{itemize}
\item[(1)] $\mathcal{F}_c$ has a cutoff in total variation.

\item[(2)] For $\delta\in(0,1)$, $\mathcal{F}_\delta$ has a cutoff in total variation.

\item[(3)] $\mathcal{F}_c$ has precutoff in total variation.

\item[(4)] For $\delta\in(0,1)$, $\mathcal{F}_\delta$ has a precutoff in total variation.

\item[(5)] $t_n/\ell_n\ra\infty$.
\end{itemize}
\end{thm}

The above theorem is obvious from Theorems \ref{t-csal12-2}, \ref{t-tvmixingcn} and \ref{t-gap}. We use two classical examples, simple random walks and Ehrenfest chains, to illustrate how to apply Theorem \ref{t-cutoffrates} to determine the total variation cutoff and mixing times.

\begin{ex}[Simple random walks on finite paths]\label{ex-srw}
For $n\ge 1$, the simple random walk on $\{0,...,n\}$ is a birth and death chain with $p_{n,i}=q_{n,i+1}=1/2$ for $0\le i<n$ and $r_{n,0}=r_{n,n}=1/2$. It is clear that $K_n$ is irreducible and aperiodic with uniform stationary distribution. Let $t_n,\ell_n$ be the constants in Theorem \ref{t-cutoffrates}. It is an easy exercise to show that $\ell_n\asymp n^2\asymp t_n$. By Theorem \ref{t-cutoffrates}, neither $\mathcal{F}_c$ nor $\mathcal{F}_\delta$ has total variation precutoff, but $T_{n,\text{\tiny TV}}^{(c)}(\epsilon)\asymp n^2\asymp T_{n,\text{\tiny TV}}^{(\delta)}(\epsilon)$ for $\epsilon\in(0,1/2)$ and $\delta\in(0,1)$. In fact, one may use a hitting time statement to prove that the mixing time has order at least $n^2$, when $\epsilon\in[1/2,1)$. This implies that the above approximation of mixing time holds for $\epsilon\in(0,1)$.
\end{ex}

\begin{ex}[Ehrenfest chains]\label{ex-Ehrenfest}
Consider the Ehrenfest chain on $\{0,...,n\}$, which is a birth and death chain with rates $p_{n,i}=1-i/n$ and $q_{n,i}=i/n$. It is obvious that $K_n$ is irreducible and periodic with stationary distribution $\pi_n(i)=2^{-n}\binom{n}{i}$. An application of the representation theory shows that, for $0\le i\le n$, $2i/n$ is an eigenvalue of $I-K_n$. Let $\lambda_n,s_n$ be the constants in Theorem \ref{t-csal12-2}. Clearly, $\lambda_n=2/n$ and $s_n\asymp n\log n$ and, by Theorem \ref{t-csal12-2}, both $\mathcal{F}_c$ and $\mathcal{F}_\delta$ have a total variation cutoff. Note that, as a simple corollary, one obtains the non-trivial estimates
\[
 \sum_{i=0}^{\lceil\frac{n}{2}\rceil-1}\frac{\binom{n}{0}+\cdots+\binom{n}{i}}{\binom{n}{i}}\asymp n\log n,\quad \max_{0\le i<n/2}\sum_{j=0}^i\binom{n}{j}\times\sum_{j=i}^{\lceil\tfrac{n}{2}\rceil-1}\binom{n}{i}^{-1}\asymp n.
\]
For a detailed computation on the total variation and the $L^2$-distance, see e.g. \cite{D88}.
\end{ex}

In the next subsections, we consider birth and death chains of special types.

\subsection{Chains with valley stationary distributions}
In this subsection, we consider birth and death chains with valley stationary distribution. For $n\ge 1$, let $\Omega_n=\{0,1,...,n\}$ and $K_n$ be an irreducible birth and death chain on $\Omega_n$ with birth, death and holding rates, $p_{n,i},q_{n,i},r_{n,i}$. Suppose that there is $j_n\in\Omega_n$ such that
\begin{equation}\label{eq-valley}
 p_{n,i}\le q_{n,i+1},\,\forall i<j_n,\quad p_{n,i}\ge q_{n,i+1},\,\forall i\ge j_n.
\end{equation}
Obviously, the stationary distribution $\pi_n$ of $K_n$ satisfies $\pi_n(i)\ge\pi_n(i+1)$ for $i<j_n$ and $\pi_n(i)\le\pi_n(i+1)$ for $i\ge j_n$.

Let $t_n,\ell_n$ be the constants in Theorem \ref{t-cutoffrates} and write
\[
\ell_n=\max\left\{\max_{j:j<j_n}\sum_{k=j+1}^{j_n}\frac{\pi_n([0,j])}{\pi_n(k)q_{n,k}},\max_{j:j>j_n}\sum_{k=j_n}^{j-1}\frac{\pi_n([j,n])}{\pi_n(k)p_{n,k}}\right\}.
\]
Set
\[
 M_L=\max_{0<i\le j_n}q_{n,i},\,\, m_L=\min_{0<i\le j_n}q_{n,i},\,\, M_R=\max_{j_n\le i<n}p_{n,i},\,\, m_R=\min_{j_n\le i<n}p_{n,i}.
\]
Clearly,
\[
 \ell_n\le\max\left\{\frac{\pi_n([0,j_n])}{m_L}\sum_{i=0}^{j_n}\frac{1}{\pi_n(i)},\frac{\pi_n([j_n,n])}{m_R}\sum_{i=j_n}^n\frac{1}{\pi_n(i)}\right\}.
\]
Let $j_n'$ be such that $\pi_n([0,j_n'])\ge\pi_n([0,j_n])/2$ and $\pi_n([j_n',j_n])\ge\pi_n([0,j_n])/2$. Note that if $j_n\ge 1$, then  $j_n\ge\max\{2j_n',j_n'+1\}$. By (\ref{eq-valley}), this implies
\[
 \sum_{k=j_n'+1}^{j_n}\frac{\pi_n([0,j_n'])}{\pi_n(k)}\ge\frac{\pi_n([0,j_n])}{4}\sum_{k=j_n'}^{j_n}\frac{1}{\pi_n(k)}\ge\frac{\pi_n([0,j_n])}{8}\sum_{k=0}^{j_n}\frac{1}{\pi_n(k)}.
\]
One can derive a similar inequality from the other end point and this yields
\[
 \ell_n\ge\frac{1}{8}\min\left\{\frac{\pi_n([0,j_n])}{M_L}\sum_{i=0}^{j_n}\frac{1}{\pi_n(i)},\frac{\pi_n([j_n,n])}{M_R}\sum_{i=j_n}^n\frac{1}{\pi_n(i)}\right\}.
\]
For $t_n$, note that
\[
 \frac{\pi_n([0,j_n-1])}{2}\sum_{k=0}^{j_n-1}\frac{1}{\pi_n(k)}\le\sum_{k=0}^{j_n-1}\frac{\pi_n([0,k])}{\pi_n(k)}\le\pi_n([0,j_n-1])\sum_{k=0}^{j_n-1}\frac{1}{\pi_n(k)}
\]
and
\[
 \frac{\pi_n([j_n+1,n])}{2}\sum_{k=j_n+1}^{n}\frac{1}{\pi_n(k)}\le\sum_{k=j_n+1}^{n}\frac{\pi_n([k,n])}{\pi_n(k)}\le\pi_n([j_n+1,n])\sum_{k=j_n+1}^{n}\frac{1}{\pi_n(k)}
\]
This implies
\[
 t_n\le\max\left\{\frac{\pi_n([0,j_n])}{m_L}\sum_{i=0}^{j_n}\frac{1}{\pi_n(i)},\frac{\pi_n([j_n,n])}{m_R}\sum_{i=j_n}^n\frac{1}{\pi_n(i)}\right\}
\]
and
\[
 t_n\ge\frac{1}{8}\max\left\{\frac{\pi_n([0,j_n])}{M_L}\sum_{i=0}^{j_n}\frac{1}{\pi_n(i)},\frac{\pi_n([j_n,n])}{M_R}\sum_{i=j_n}^n\frac{1}{\pi_n(i)}\right\}
\]
The following theorem is an immediate consequence of the above discussion and Theorem \ref{t-cutoffrates}.

\begin{thm}\label{t-valley}
Let $\mathcal{F}=\{(\Omega_n,K_n,\pi_n)|n=1,2,...\}$ be a family of birth and death chains satisfying \textnormal{(\ref{eq-valley})}. Assume that $\pi_n([0,j_n])\asymp\pi_n([j_n,n])$ and
\[
 \max_{0<i\le j_n}q_{n,i}\asymp\min_{0<i\le j_n}q_{n,i},\quad \max_{j_n\le i<n}p_{n,i}\asymp\min_{j_n\le i<n}p_{n,i}.
\]
Then, there is no cutoff for $\mathcal{F}_c,\mathcal{F}_\delta$ and, for $\epsilon\in(0,1/2)$ and $\delta\in(0,1)$,
\[
 T_{n,\textnormal{\tiny TV}}^{(c)}(\epsilon)\asymp T_{n,\textnormal{\tiny TV}}^{(\delta)}(\epsilon)\asymp\frac{1}{\lambda_n}\asymp \max\left\{\frac{1}{q_{n,j_n}}\sum_{i=0}^{j_n}\frac{1}{\pi_n(i)},\frac{1}{p_{n,j_n}}\sum_{i=j_n}^n\frac{1}{\pi_n(i)}\right\}.
\]
\end{thm}

For an illustration of the above theorem, we consider the following Markov chains. For $n\ge 1$, let $\Omega_n=\{0,1,...,n\}$, $\pi_n$ be a non-uniform probability distribution on $\Omega_n$ satisfying (\ref{eq-valley}) and $M_n$ be a transition matrix given by
\begin{equation}\label{eq-Met}
 M_n(i,j)=\begin{cases}1/2&\text{for }j=i-1,i\le j_n,\\1/2&\text{for }j=i+1,i\ge j_n,\\\pi_n(i+1)/(2\pi_n(i))&\text{for }j=i+1,i<j_n,\\\pi_n(i-1)/(2\pi_n(i))&\text{for }j=i-1,i>j_n,\\1/2-\pi_n(i+1)/(2\pi_n(i))&\text{for }j=i<j_n,\\1/2-\pi_n(i-1)/(2\pi_n(i))&\text{for }j=i>j_n.\end{cases}
\end{equation}
Note that $M_n$ is the Metropolis chain for $\pi_n$ associated to the simple random walk on $\Omega_n$. For more information on the Metropolis chain, see \cite{DS98} and the references therein. The next theorem is a corollary of Theorem \ref{t-valley}.

\begin{thm}\label{t-met}
Let $\mathcal{F}=\{(\Omega_n,M_n,\pi_n)|n=1,2,..\}$ be the family of Metropolis chains satisfying \textnormal{(\ref{eq-valley})-(\ref{eq-Met})}. Suppose $\pi_n([0,j_n])\asymp \pi_n([j_n,n])$. Then, neither $\mathcal{F}_c$ nor $\mathcal{F}_\delta$ has a total variation precutoff but, for $\epsilon\in(0,1/2)$ and $\delta\in(0,1)$,
\[
 T_{n,\textnormal{\tiny TV}}^{(c)}(\epsilon)\asymp \sum_{i=0}^n\frac{1}{\pi_n(i)}\asymp T_{n,\textnormal{\tiny TV}}^{(\delta)}(\epsilon).
\]
\end{thm}

\begin{ex}
Let $a>0$ and $\check{\pi}_{n,a},\hat{\pi}_{n,a}$ be probability measures on $\{0,\pm 1,...,\pm n\}$ given by
\begin{equation}\label{eq-checkhat}
 \check{\pi}_{n,a}(i)=\check{c}_{n,a}(|i|+1)^a,\quad\hat{\pi}_{n,a}(i)=\hat{c}_{n,a}(n-|i|+1)^a,
\end{equation}
where $\check{c}_{n,a},\hat{c}_{n,a}$ are normalizing constants. Let $\check{\mathcal{F}},\hat{\mathcal{F}}$ be families of the Metropolis chains for $\check{\pi}_{n,a},\hat{\pi}_{n,a}$ associated to the simple random walks on $\{0,\pm 1,...,\pm n\}$, that is,
\[
 \check{M}_{n,a}(i,j)=\check{M}_{n,a}(-i,-j),\quad \hat{M}_{n,a}(i,j)=\hat{M}_{n,a}(-i,-j)
\]
and
\[
 \check{M}_{n,a}(i,j)=\begin{cases}\frac{1}{2}&\text{if }j=i+1,i\in[0,n-1]\\\frac{i^a}{2(i+1)^a}&\text{if }j=i-1,i\in[1,n]\\\frac{(i+1)^a-i^a}{2(i+1)^a}&\text{if }j=i,i\notin\{0,n\}\\1-\frac{n^a}{2(n+1)^a}&\text{if }i=j=n\end{cases}
\]
and
\[
 \hat{M}_{n,a}(i,j)=\begin{cases}\frac{1}{2}&\text{if }j=i-1,i\in[1,n]\\\frac{(n-i)^a}{2(n-i+1)^a}&\text{if }j=i+1,i\in[0,n-1]\\\frac{(n-i+1)^a-(n-i)^a}{2(n-i+1)^a}&\text{if }j=i\ne 0\\1-\frac{n^a}{(n+1)^a}&\text{if }i=j=0\end{cases}.
\]

Let $\check{\lambda}_{n,a},\hat{\lambda}_{n,a}$ and $\check{T}_{n,a},\hat{T}_{n,a}$ be the spectral gaps and total variation mixing times of $\check{M}_{n,a},\hat{M}_{n,a}$. It has been proved in \cite{CSal12,SC99} that there is $C>1$ such that, for all $a>0$ and $n\ge 1$,
\[
 \frac{1}{C\check{\lambda}_{n,a}}\asymp n^a\left(\left(1+\frac{1}{n}\right)^a+\frac{n}{1+a}\right)(1+v(n,a))\le\frac{C}{\check{\lambda}_{n,a}}
\]
and
\[
 \frac{1}{C\hat{\lambda}_{n,a}}\le\frac{(n+a)^2}{(1+a)^2}\le\frac{C}{\hat{\lambda}_{n,a}},
\]
where $v(n,1)=\log n$ and $v(n,a)=(n^{1-a}-1)/(1-a)$ for $a\ne 1$. By Theorem \ref{t-valley},  $\check{\mathcal{F}}_c$ and $\check{\mathcal{F}}_\delta$ have no cutoff in total variation but, for fixed $a>0$, $\epsilon\in(0,1/2)$ and $\delta\in(0,1)$,
\[
 \check{T}_{n,a}^{(c)}(\epsilon)\asymp\check{T}_{n,a}^{(\delta)}(\epsilon)\asymp\begin{cases}n^2&\text{if }a\in(0,1)\\n^2\log n&\text{if }a=1\\n^{1+a}&\text{if }a\in(1,\infty)\end{cases}.
\]
The above result in continuous time case is also obtained in \cite{SC99}.

To see the cutoff for $\hat{\mathcal{F}}$, let
\[
 t_n=\sum_{k=0}^{n-1}\frac{\hat{\pi}_{n,a}([-n,-n+k])}{\hat{\pi}_{n,a}(-n+k)}=\sum_{k=1}^nk^{-a}\sum_{j=1}^kj^a.
\]
By Theorems \ref{t-upper}-\ref{t-lower}, we have
\[
 \frac{2t_n}{3}\le\hat{T}_{n,a}^{(c)}(1/10)\le 3600t_n.
\]
Note that, for $k\ge 1$ and $a>0$,
\[
 \frac{k^a(k+a)}{2(1+a)}\le \sum_{j=1}^kj^a\le \frac{2k^a(k+a)}{1+a}.
\]
This implies
\[
 \frac{n(n+a)}{6(1+a)}\le\hat{T}_{n,a}^{(c)}(1/10)\le \frac{14400n(n+a)}{1+a}.
\]
We collect the above results in the following theorem.

\begin{thm}\label{t-checkhat}
For $n\ge 1$, let $a_n>0$ and $\check{\pi}_{n,a_n},\hat{\pi}_{n,a_n}$ be probability measures given by \textnormal{(\ref{eq-checkhat})}. Let $\check{\mathcal{F}},\hat{\mathcal{F}}$ be the families of Metropolis chains for $\check{\pi}_{n,a_n},\hat{\pi}_{n,a_n}$ as above with total variation mixing time $\check{T}_{n,\textnormal{\tiny TV}},\hat{T}_{n,\textnormal{\tiny TV}}$. Then, for $\epsilon\in(0,1/2)$ and $\delta\in(0,1)$,
\[
 \hat{T}_{n,\textnormal{\tiny TV}}^{(c)}(\epsilon)\asymp\hat{T}_{n,\textnormal{\tiny TV}}^{(\delta)}(\epsilon)\asymp\frac{n(n+a_n)}{1+a_n}
\]
and
\[
 \check{T}_{n,\textnormal{\tiny TV}}^{(c)}(\epsilon)\asymp\check{T}_{n,\textnormal{\tiny TV}}^{(\delta)}(\epsilon)\asymp n^{a_n}\left(\left(1+\frac{1}{n}\right)^{a_n}+\frac{n}{1+a_n}\right)(1+v(n,a_n)),
\]
where $v(n,1)=\log n$ and $v(n,a)=(n^{1-a}-1)/(1-a)$ for $a\ne 1$.

Moreover, neither $\check{\mathcal{F}}_c$ nor $\check{\mathcal{F}}_\delta$ has a total variation cutoff. Also, $\hat{\mathcal{F}}_c$ and $\hat{\mathcal{F}}_\delta$ have a total variation cutoff if and only if $a_n\ra\infty$.
\end{thm}
\end{ex}

\subsection{Chains with monotonic stationary distributions}
In this subsection, we consider birth and death chains with monotonic stationary distributions. For $n\ge 1$, let $\Omega_n=\{0,1,...,n\}$ and $K_n$ be a birth and death chain on $\Omega_n$ with birth, death and holding rates, $p_{n,i},q_{n,i},r_{n,i}$. Suppose that
\begin{equation}\label{eq-mono}
 p_{n,i}\ge q_{n,i+1},\quad\forall 0\le i<n.
\end{equation}
If $K_n$ is irreducible, then the stationary distribution $\pi_n$ satisfying $\pi_n(i)\le\pi_n(i+1)$ for $0\le i<n$. Let $j_n\in\Omega_n$ and $t_n,\ell_n$ be the constants in Theorem \ref{t-cutoffrates}. Assume that $\pi_n([0,j_n])\asymp\pi_n([j_n,n])$ and
\begin{equation}\label{eq-mono2}
 \max_{0\le i<j_n}p_{n,i}\asymp \min_{0\le i<j_n}p_{n,i},\quad \max_{j_n\le i<n}p_{n,i}\asymp \min_{j_n\le i<n}p_{n,i}.
\end{equation}
Using a discussion similar to that in front of Theorem \ref{t-valley}, one can show that
\[
 t_n\asymp\max\left\{\frac{1}{p_{n,1}}\sum_{k=0}^{j_n-1}\frac{\pi_n([0,k])}{\pi_n(k)},\frac{1}{p_{n,j_n}}\sum_{k=j_n}^n\frac{1}{\pi_n(k)}\right\}
\]
and
\[
 \ell_n\asymp\max\left\{\frac{1}{p_{n,1}}\max_{0\le j<j_n}\sum_{k=j}^{j_n-1}\frac{\pi_n([0,j])}{\pi_n(k)},\frac{1}{p_{n,j_n}}\sum_{k=j_n}^n\frac{1}{\pi_n(k)}\right\}.
\]
This leads to the following theorem.

\begin{thm}\label{t-mono}
Let $\mathcal{F}=\{(\Omega_n,K_n,\pi_n)|n=1,2,...\}$ be a family of irreducible birth and death chains with $\Omega_n=\{0,1,...,n\}$ and birth, death and holding rates $p_{n,i},q_{n,i},r_{n,i}$. Let $\lambda_n,T_{n,\textnormal{\tiny TV}}$ be the spectral gap and total variation mixing time of $K_n$ and set
\[
 u_n=\sum_{k=0}^{j_n-1}\frac{\pi_n([0,k])}{\pi_n(k)},\quad v_n=\max_{0\le j<j_n}\sum_{k=j}^{j_n-1}\frac{\pi_n([0,j])}{\pi_n(k)},\quad w_n=\sum_{k=j_n}^n\frac{1}{\pi_n(k)}.
\]
Assume that $\pi_n([0,j_n])\asymp \pi_n([j_n,n])$ and \textnormal{(\ref{eq-mono2})} holds. Then, for $\epsilon\in(0,1/2)$ and $\delta\in(0,1)$,
\[
 \lambda_n^{-1}\asymp\max\left\{\frac{v_n}{p_{n,1}},\frac{w_n}{p_{n,j_n}}\right\},\quad T_{n,\textnormal{\tiny TV}}^{(c)}(\epsilon)\asymp T_{n,\textnormal{\tiny TV}}^{(\delta)}(\epsilon)\asymp \max\left\{\frac{u_n}{p_{n,1}},\frac{w_n}{p_{n,j_n}}\right\}.
\]
Moreover, $\mathcal{F}_c$ and $\mathcal{F}_\delta$ have a total variation cutoff if and only if
\[
 u_n/v_n\ra\infty,\quad (u_np_{n,j_n})/(w_np_{n,1})\ra\infty.
\]
\end{thm}

For $n\ge 1$, let $f_n$ be a non-decreasing function on $[0,n]$ and set $F_n(x)=\int_0^xf_n(t)dt$ and $G_n(x,m)=\int_x^{m}1/f_n(t)dt$. Note that if there is $C>1$ such that
\[
 C^{-1}f_n(i)\pi_n(0)\le \pi_n(i)\le Cf_n(i)\pi_n(0),\quad\forall 0\le i\le n,\,n\ge 1,
\]
then
\[
 \frac{1}{2C^2}\left(\frac{F_n(k)}{f_n(k)}+1\right)\le \frac{\pi_n([0,k])}{\pi_n(k)}\le C^2\left(\frac{F_n(k)}{f_n(k)}+1\right)
\]
and
\[
 \frac{1}{2C}\left(G_n(j,j_n)+\frac{1}{f_n(j)}\right)\le\pi_n(0)\sum_{k=j}^{j_n-1}\frac{1}{\pi_n(k)}\le C\left(G_n(j,j_n)+\frac{1}{f_n(j)}\right).
\]
This implies
\[
 \pi_n([0,j])\sum_{k=j}^{j_n-1}\frac{1}{\pi_n(k)}\le C^2\left(G_n(j,j_n)+\frac{1}{f_n(j)}\right)\left(F_n(j)+f_n(j)\right)
\]
and
\[
 \pi_n([0,j])\sum_{k=j}^{j_n-1}\frac{1}{\pi_n(k)}\ge \frac{1}{4C^2}\left(G_n(j,j_n)+\frac{1}{f_n(j)}\right)\left(F_n(j)+f_n(j)\right).
\]

Let $u_n,v_n,w_n$ be the constants in Theorem \ref{t-mono} and assume that
\[
 \min_{0\le i<n}p_{n,i}\asymp \max_{0\le i<n}p_{n,i}\asymp 1.
\]
Consider the following cases.

{\bf Case 1:} $f_n(x)=\exp\{\alpha_nx^{\beta_n}\}$ with $\inf_n\alpha_n>0$ and $\inf_n\beta_n\ge 1$. In this case, $F_n(x)=O(f_n(x))$ and $G_n(x,m)=O(1/f_n(x))$ for $1\le x<m$. By setting $j_n=n$, we obtain
\[
 \pi_n([0,j_n])\asymp\pi_n([j_n,n]), \quad u_n\asymp n,\quad v_n\asymp w_n\asymp 1.
\]
By Theorem \ref{t-mono}, $\lambda_n\asymp 1$ and, for $\epsilon\in(0,1/2)$ and $\delta\in(0,1)$,
\[
 T_{n,\text{\tiny TV}}^{(c)}(\epsilon)\asymp T_{n,\text{\tiny TV}}^{(\delta)}(\epsilon)\asymp n.
\]
There is a total variation cutoff for $\mathcal{F}_c$ or $\mathcal{F}_\delta$.

{\bf Case 2:} $f_n(x)=\exp\{\alpha_nx^{\beta_n}\}$ with $0<\inf_n\alpha_n\le\sup_n\alpha_n<\infty$ and $0<\inf_n\beta_n\le\sup_n\beta_n<1$.
Note that, for $\alpha\in\mathbb{R}$ and $\beta\in(0,1)$,
\[
 \frac{d}{dx}\left(x^{1-\beta}e^{\alpha x^{\beta}}\right)=\left(\alpha\beta+(1-\beta)x^{-\beta}\right)e^{\alpha x^\beta}.
\]
This implies that, uniformly for $n/2\le x$ and $1+x\le m\le n$,
\[
 F_n(x)\asymp x^{1-\beta_n}f_n(x),\quad G_n(x,m)\asymp\left(\frac{x^{1-\beta_n}}{f_n(x)}-\frac{m^{1-\beta_n}}{f_n(m)}\right).
\]
Letting $j_n=\lfloor n-n^{1-\beta_n}\rfloor$ yields
\[
 \pi_n([0,j_n])\asymp\pi_n([j_n,n]),\quad u_n\asymp n^{2-\beta_n},\quad v_n\asymp n^{2-2\beta_n}\asymp w_n.
\]
By Theorem \ref{t-mono}, $\mathcal{F}_c$ and $\mathcal{F}_\delta$ have a total variation cutoff and
\[
 \lambda_n\asymp n^{2\beta_n-2},\quad T_{n,\text{\tiny TV}}^{(c)}(\epsilon)\asymp T_{n,\text{\tiny TV}}^{(\delta)}(\epsilon)\asymp n^{2-\beta_n},\quad\forall \epsilon\in(0,1/2),\,\delta\in(0,1).
\]

{\bf Case 3:} $f_n(x)=\exp\{\alpha_n[\log(x+1)]^{\beta_n}\}$ with $0<\inf_n\alpha_n\le\sup_n\alpha_n<\infty$ and $1<\inf_n\beta_n\le\sup_n\beta_n<\infty$. Note that, for $\alpha\in\mathbb{R}$ and $\beta>1$,
\[
 \frac{d}{dx}\left(\frac{(x+1)e^{\alpha[\log(x+1)]^\beta}}{[\log(x+1)]^{\beta-1}}\right)=
\left(\alpha\beta+\frac{1-(\beta-1)/\log(x+1)}{[\log(x+1)]^{\beta-1}}\right)e^{\alpha[\log(x+1)]^\beta}.
\]
This implies that, uniformly for $n/2\le x<m\le n$,
\[
 F_n(x)\asymp\frac{(x+1)}{[\log(x+1)]^{\beta_n-1}}e^{\alpha_n[\log(x+1)]^{\beta_n}}
\]
and
\[
 G_n(x,m)\asymp\left(\frac{(x+1)e^{-\alpha_n[\log(x+1)]^{\beta_n}}}{[\log(x+1)]^{\beta_n-1}}
 -\frac{(m+1)e^{-\alpha_n[\log(m+1)]^{\beta_n}}}{[\log(m+1)]^{\beta_n-1}}\right).
\]
Set $j_n=n[1-(\log n)^{1-\beta_n}]$. The above computation leads to
\[
 \pi_n([0,j_n])\asymp\pi_n([j_n,n]),\quad u_n\asymp n^2(\log n)^{1-\beta_n},\quad v_n\asymp n^2(\log n)^{2-2\beta_n}\asymp w_n.
\]
By Theorem \ref{t-mono}, both $\mathcal{F}_c$ and $\mathcal{F}_\delta$ have a total variation cutoff and, for $\epsilon\in(0,1/2)$ and $\delta\in(0,1)$,
\[
 \lambda_n\asymp n^{-2}(\log n)^{2\beta_n-2},\quad T_{n,\text{\tiny TV}}^{(c)}(\epsilon)\asymp n^2(\log n)^{1-\beta_n}\asymp T_{n,\text{\tiny TV}}^{(\delta)}(\epsilon).
\]

{\bf Case 4:} $f_n(x)=\exp\{\alpha_n[\log(x+1)]^{\beta_n}\}$ with $\sup_n\alpha_n<\infty$ and $\sup_n\beta_n\le 1$. Note that, as a consequence of the mean values theorem, one may choose, for each $0<a<1$, a constant $b\in(a,1)$ such that
\begin{align}
 1<\frac{f_n(n)}{f_n(an)}&=\exp\left\{\alpha_n\left[\left(\log(n+1)\right)^{\beta_n}
 -\left(\log(an+1)\right)^{\beta_n}\right]\right\}\notag\\
 &=\exp\left\{\alpha_n\beta_n(1-a)n\frac{\left(\log(bn+1)\right)^{\beta_n-1}}{bn+1}\right\}\notag\\
 &\le\exp\left\{\frac{1-a}{a}\sup_n\alpha_n\right\}<\infty.\notag
\end{align}
This implies that, for $a\in(0,1)$, one may choose a constant $A>1$ (depending on $a$) such that
\[
 \frac{1}{An}\le\pi_n(x)\le\frac{A}{n},\quad\forall x\ge an,\,\,n\ge 1.
\]
Choosing $j_n=\lfloor n/2\rfloor$ yields $\pi_n([0,j_n])\asymp\pi_n([j_n,n])$ and $u_n\asymp v_n\asymp w_n\asymp n^2$. By Theorem \ref{t-mono}, there is no total variation cutoff for $\mathcal{F}_c$ or $\mathcal{F}_\delta$ and
\[
 T_{n,\text{\tiny TV}}^{(c)}(\epsilon)\asymp T_{n,\text{\tiny TV}}^{(\delta)}(\epsilon)\asymp\lambda_n^{-1}\asymp n^2,\quad\forall \epsilon\in(0,1/2),\,\delta\in(0,1).
\]

\subsection{Chains with symmetric stationary distributions}
This subsection is dedicated to the study of birth and death chains with symmetric stationary distributions. Let $K$ be an irreducible birth and death chain on $\{0,...,n\}$ with stationary distribution $\pi$. Note that $\pi$ is symmetric at $n/2$, that is, $\pi(n-i)=\pi(i)$ for $0\le i\le n/2$, if and only if
\[
 p_ip_{n-i-1}=q_{i+1}q_{n-i},\quad\forall 0\le i\le n/2.
\]
By the symmetry of $\pi$, we will fix $j_n=\lfloor n/2\rfloor$ when applying Theorem \ref{t-cutoffrates}.

Consider a family of irreducible birth and death chains, $\mathcal{F}=\{(\Omega_n,K_n,\pi_n)|n=1,2,...\}$ with $\Omega_n=\{0,1,...,n\}$. Let $p_{n,i},q_{n,i},r_{n,i}$ be respectively the birth, death and holding rates of $K_n$ and $t_n,\ell_n$ be constants in Theorem \ref{t-cutoffrates}. Assume that $\pi_n$ is symmetric at $n/2$. Continuously using the fact $(a+b)/2\le \max\{a,b\}\le a+b$ for $a\ge 0,b\ge 0$, we obtain
\[
  t_n\asymp\sum_{k:k\le n/2}\frac{\pi_n([0,k])}{\pi_n(k)\min\{p_{n,k},q_{n,n-k}\}}
\]
and
\[
 \ell_n\asymp\max_{j:j\le n/2}\sum_{k:j\le k\le n/2}\frac{\pi_n([0,j])}{\pi_n(k)\min\{p_{n,k},q_{n,n-k}\}}.
\]
Theorem \ref{t-cutoffrates} can be rewritten as follows.

\begin{thm}\label{t-uniform}
Let $\mathcal{F}=\{(\Omega_n,K_n,\pi_n)|n=1,2,...\}$ be a family of irreducible birth and death chains with $\Omega_n=\{0,1,...,n\}$. Let $\lambda_n$ and $p_{n,i},q_{n,i},r_{n,i}$ be the spectral gap and the birth, death and holding rates of $K_n$. Assume that
\[
 p_{n,i}p_{n,n-i-1}=q_{n,i+1}q_{n,n-i},\quad \forall 0\le i\le n/2.
\]
Then, for $\epsilon\in(0,1/2)$ and $\delta\in(0,1)$,
\[
 \lambda_n\asymp 1/\ell_n, \quad T_{n,\textnormal{\tiny TV}}^{(c)}(\epsilon)\asymp T_{n,\textnormal{\tiny TV}}^{(\delta)}(\epsilon)\asymp t_n,
\]
where
\[
 t_n=\sum_{k:k\le n/2}\frac{\pi_n([0,k])}{\pi_n(k)\min\{p_{n,k},q_{n,n-k}\}}
\]
and
\[
 \ell_n=\max_{j:j\le n/2}\left\{\pi_n([0,j])\sum_{k:j\le k\le n/2}\frac{1}{\pi_n(k)\min\{p_{n,k},q_{n,n-k}\}}\right\}.
\]
Moreover, the following are equivalent.
\begin{itemize}
\item[(1)] $\mathcal{F}_c$ has a cutoff in total variation.

\item[(2)] For $\delta\in(0,1)$, $\mathcal{F}_\delta$ has a cutoff in total variation.

\item[(3)] $\mathcal{F}_c$ has a precutoff in total variation.

\item[(4)] For $\delta\in(0,1)$, $\mathcal{F}_\delta$ has a precutoff in total variation.

\item[(5)] $t_n/\ell_n\ra\infty$.
\end{itemize}
\end{thm}

The next theorem considers a perturbation of birth and death chains which has the same stationary distribution as the original chains. The new chains keep the order of mixing time and spectral gap unchanged.
\begin{thm}\label{t-uniform2}
Consider the family in Theorem \ref{t-uniform} and assume that
\[
 p_{n,i}p_{n,n-i-1}=q_{n,i+1}q_{n,n-i},\quad \forall 0\le i\le n/2.
\]
For $n\ge 1$, let $A_n\subset\{0,...,n-1\}$, $c_{n,i}\in[0,1]$ for $i\in A_n$ and $\widetilde{K}_n$ be a birth and death chain on $\Omega_n$ with birth and death rates, $\widetilde{p}_{n,i},\widetilde{q}_{n,i}$, satisfying
\[
 \begin{cases}\widetilde{p}_{n,i}=c_{n,i}p_{n,i}+(1-c_{n,i})\min\{p_{n,i},q_{n,n-i}\}&\text{for }i\in A_n,\\
\widetilde{q}_{n,i+1}=q_{n,i+1}\widetilde{p}_{n,i}/p_{n,i}&\text{for }i\in A_n,\\
\widetilde{p}_{n,i}=p_{n,i},\quad \widetilde{q}_{n,i+1}=q_{n,i+1}&\text{for }i\notin A_n.\end{cases}
\]
Let $\lambda_n,\widetilde{\lambda}_n$ and $T_{n,\textnormal{\tiny TV}}(\epsilon),\widetilde{T}_{n,\textnormal{\tiny TV}}(\epsilon)$ be the spectral gaps and total variation mixing times of $K_n,\widetilde{K}_n$. Then, given $\epsilon\in(0,1/2)$ and $\delta\in(0,1)$,
\[
 \widetilde{\lambda}_n\asymp\lambda_n,\quad\widetilde{T}^{(c)}_{n,\textnormal{\tiny TV}}(\epsilon)\asymp T^{(c)}_{n,\textnormal{\tiny TV}}(\epsilon)\asymp\widetilde{T}^{(\delta)}_{n,\textnormal{\tiny TV}}(\epsilon)\asymp T^{(\delta)}_{n,\textnormal{\tiny TV}}(\epsilon),
\]
where the approximation is uniform on the choice of $A_n,c_{n,i}$.
\end{thm}
\begin{proof}
The approximation of the spectral gap and the total variation mixing time is immediate from Theorem \ref{t-uniform}, whereas the uniformity of the approximation is given by Theorems \ref{t-upper}, \ref{t-lower} and \ref{t-gap}.
\end{proof}

\begin{ex}
For $n\ge 1$, let $K_n$ be a birth and death chain on $\{0,1,...,2n\}$ given by
\[
 K_n(i,i+1)=K_n(i+1,i)=\begin{cases}1/2&\text{for even $i$}\\1/(2n)&\text{for odd $i$}\end{cases}.
\]
By Theorem \ref{t-uniform2}, the mixing time and spectral gap of $K_n$ are comparable with those of $\widetilde{K}_n$, where $\widetilde{K}_n(i,i+1)=\widetilde{K}_n(i+1,i)=1/(2n)$ for $0\le i<2n$. Let $\mathcal{F}$ be the family consisting of $K_n$. By Theorem \ref{t-uniform}, neither $\mathcal{F}_c$ nor $\mathcal{F}_\delta$ has a total variation precutoff and $T_{n,\text{\tiny TV}}^{(c)}(\epsilon)\asymp T_{n,\text{\tiny TV}}^{(\delta)}(\epsilon)\asymp \lambda_n^{-1}\asymp n^3$ for all $\epsilon\in(0,1/2)$ and $\delta\in(0,1)$, which is nontrivial.
\end{ex}

Next, we consider simple random walks on finite paths with bottlenecks. For $n\ge 1$, let $k_n\le n$ and $x_{n,1},...,x_{n,k_n}$ be positive integers satisfying $1\le x_{n,i}<x_{n,i+1}\le n$ for $i=1,...,k_n-1$. Let $K_n$ be the birth and death chain on $\{0,1,...,n\}$ of which birth, death and holding rates are given by
\begin{equation}\label{eq-btnk}
 p_{n,i-1}=q_{n,i}=\begin{cases}1/2&\text{for }i\notin\{x_{n,1},...,x_{n,k_n}\}\\\epsilon_{n,j}&\text{for }i=x_{n,j},\,1\le j\le k_n\end{cases},
\end{equation}
where $\epsilon_{n,j}\in(0,1/2]$ for $1\le j\le k_n$. Clearly, $K_n$ is irreducible and the stationary distribution, say $\pi_n$, is uniform on $\{0,1,...,n\}$. The following theorem  is immediate from Theorems \ref{t-uniform}.

\begin{thm}\label{t-btnk}
Let $\mathcal{F}$ be a family of birth and death chains given by \textnormal{(\ref{eq-btnk})} and $\lambda_n$ be the spectral gap of $K_n$. For $n\ge 1$, set
\[
 t_n=n^2+\sum_{i=1}^{k_n}\frac{\min\{x_{n,i},n+1-x_{n,i}\}}{\epsilon_{n,i}}
\]
and
\[
 \ell_n=
n^2+\max_{j:j\le n/2}\left\{\sum_{i:|x_{n,i}-n/2|\le j}\frac{n/2+1-j}{\epsilon_{n,i}}\right\}.
\]
Then, for all $\epsilon\in(0,1/2)$ and $\delta\in(0,1)$,
\[
 T_{n,\textnormal{\tiny TV}}^{(c)}(\epsilon)\asymp T_{n,\textnormal{\tiny TV}}^{(\delta)}(\epsilon)\asymp t_n,\quad \lambda_n\asymp 1/\ell_n.
\]
Furthermore, the following are equivalent.
\begin{itemize}
\item[(1)] $\mathcal{F}_c$ has a cutoff in total variation.

\item[(2)] For $\delta\in(0,1)$, $\mathcal{F}_\delta$ has a cutoff in total variation.

\item[(3)] $\mathcal{F}_c$ has precutoff in total variation.

\item[(4)] For $\delta\in(0,1)$, $\mathcal{F}_\delta$ has a precutoff in total variation.

\item[(5)] $t_n/\ell_n\ra\infty$.
\end{itemize}
\end{thm}

\begin{rem}
Let $t_n,\ell_n$ be the constants in Theorem \ref{t-btnk}. Then,
\[
 t_n\asymp n^2+\sum_{j\in L_n}\frac{x_{n,j}}{\epsilon_{n,j}}+\sum_{j\in R_n}\frac{n+1-x_{n,j}}{\epsilon_{n,j}}
\]
and
\[
 \ell_n\asymp n^2+\max_{i\in L_n}\sum_{j\in L_n:j\ge i}\frac{x_{n,i}}{\epsilon_{n,j}}+\max_{i\in R_n}\sum_{j\in R_n:j\le i}\frac{n+1-x_{n,i}}{\epsilon_{n,j}}.
\]
where $L_n=\{i:x_{n,i}\le n/2\}$ and $R_n=\{i:x_{n,i}>n/2\}$.
\end{rem}

Theorem \ref{t-btnk1} considers a special case of Theorem \ref{t-btnk} with $\epsilon_{n,i}=\epsilon_n$ for $1\le i\le k_n$. It is clear from Theorem \ref{t-btnk1} that if $k_n$ is bounded, then no cutoff exists for $\mathcal{F}_c$ or $\mathcal{F}_\delta$. The following example shows a case of cutoffs for the family in Theorem \ref{t-btnk1}.

\begin{ex}
Let $\mathcal{F}$ be the family in Theorem \ref{t-btnk1}, with $k_n=\lfloor n^{1/3}\rfloor-1$ and
\[
 x_{n,i}=\left\lfloor\frac{n^{5/6}}{n^{1/3}-i}\right\rfloor,\quad\forall 1\le i\le k_n.
\]
Clearly, for $n$ large enough, $x_{n,i}\ne x_{n,j}$ when $i\ne j$. Let $a_n,b_n$ be the constant in Theorem \ref{t-btnk1}. It is not hard to show that
\[
 a_n\asymp n^{5/6}\log n,\quad b_n\asymp n^{5/6}.
\]
By Theorem \ref{t-btnk1}, $\mathcal{F}_c$ and $\mathcal{F}_\delta$, with $\delta\in(0,1)$, have a total variation cutoff if and only if $\epsilon_n=o(n^{-7/6}\log n)$. Furtheromre, if $\epsilon_n=o(n^{-7/6}\log n)$, then
\[
 T_{n,\text{\tiny TV}}^{(c)}(\epsilon)\asymp \frac{n^{5/6}\log n}{\epsilon_n}\asymp T_{n,\text{\tiny TV}}^{(\delta)}(\epsilon),\quad \forall \epsilon,\delta\in(0,1).
\]
\end{ex}

The following two theorems treat special cases of Theorem \ref{t-btnk}.

\begin{thm}\label{t-btnk2}
Let $\mathcal{F}$ be a family of birth and death chains satisfying \textnormal{(\ref{eq-btnk})}. Let $N$ be a positive constant. Suppose, for $n\ge 1$, there are constants $J^{(n)}_1,...,J^{(n)}_N$ and a partition of $\{1,...,k_n\}$, say $I^{(n)}_1,...,I^{(n)}_N$, such that, for $1\le k\le N$,
\[
 \max_{i\in I^{(n)}_k}\{x_{n,i}\wedge (n+1-x_{n,i})\}\asymp \min_{i\in I^{(n)}_k}\{x_{n,i}\wedge (n+1-x_{n,i})\}\asymp J^{(n)}_k,
\]
where $a\wedge b=\min\{a,b\}$. Then, neither $\mathcal{F}_c$ nor $\mathcal{F}_\delta$ has a total variation cutoff. Moreover,
\[
 T_{n,\textnormal{\tiny TV}}^{(c)}(\epsilon)\asymp T_{n,\textnormal{\tiny TV}}^{(\delta)}(\epsilon)\asymp \lambda_n^{-1}\asymp t_n,\quad\forall \epsilon\in(0,1/2),\,\delta\in(0,1)
\]
where
\[
 t_n=n^2+\max_{1\le k\le N}\left\{J^{(n)}_k\sum_{l\in I^{(n)}_k}\frac{1}{\epsilon_{n,l}}\right\}.
\]
\end{thm}

The next theorem gives an example that no total variation cutoff exists for $\mathcal{F}_c,\mathcal{F}_\delta$ even when the constant $N$ in Theorem \ref{t-btnk2} tends to infinity.

\begin{thm}
Let $\mathcal{F}$ be a family of birth and death chains satisfying \textnormal{(\ref{eq-btnk})}. Suppose that $\min_j\epsilon_{n,j}\asymp\max_j\epsilon_{n,j}$ and $x_{n,i}=\lfloor in/k_n\rfloor$ with $k_n\le n/2$, then neither  $\mathcal{F}_c$ nor $\mathcal{F}_\delta$ has a total variation cutoff, but
\[
 T_{n,\textnormal{\tiny TV}}^{(c)}(\epsilon)\asymp T_{n,\textnormal{\tiny TV}}^{(\delta)}(\epsilon)\asymp \lambda_n^{-1}\asymp \max\{n^2,nk_n/\epsilon_{n,1}\},\quad\forall \epsilon\in(0,1/2),\,\delta\in(0,1).
\]
\end{thm}

\begin{rem}
Note that the assumption regarding the birth and death rates in this section can be relaxed using the comparison technique in \cite{DS93-1,DS93-2}.
\end{rem}

%We end this subsection by giving an example without total variation cutoff.
%
%\begin{ex}
%Let $\mathcal{F}$ be the family of birth and death chains given by (\ref{eq-btnk}). Suppose that, $k_n\le n/2$ and
%\[
% x_{n,i}=i,\quad\epsilon_{n,i}=\frac{1}{in},\quad\forall 1\le i\le k_n.
%\]
%Let $t_n,\ell_n$ be the constants in Theorem \ref{t-btnk}. Clearly, one has
%$t_n\asymp (n+k_n^3)n\asymp\ell_n$. This implies that both $\mathcal{F}_c$ and $\mathcal{F}_\delta$ have no total variation cutoff.
%\end{ex}

\appendix

\section{Spectral gaps of finite paths}

This section is devoted to finding the correct order of spectral gaps of finite paths. Let $G=(V,E)$ be the undirected finite graph with vertex set $V=\{0,1,2,...n\}$ and edge set $E=\{\{i,i+1\}:i=0,1,...,n-1\}$. Given two positive measures $\pi,\nu$ on $V,E$ with $\pi(V)=1$, the Dirichlet form and variance associated with $\nu$ and $\pi$ are defined by
\[
 \mathcal{E}_\nu(f,g):=\sum_{i=1}^{n-1}[f(i)-f(i+1)][g(i)-g(i+1)]\nu(i,i+1)
\]
and
\[
 \text{Var}_\pi(f):=\pi(f^2)-\pi(f)^2,
\]
where $f,g$ are functions on $V$. The spectral gap of $G$ with respect to $\pi,\nu$ is defined as
\[
 \lambda^G_{\pi,\nu}:=\min\left\{\frac{\mathcal{E}_\nu(f,f)}{\text{Var}_\pi(f)}
 \bigg|f\text{ is non-constant}\right\}.
\]

To bound the spectral gap, we need the following notations. Let $C_+(i)$ and $C_-(i)$ be constants defined by
\begin{equation}\label{eq-c+c-}
 C_+(i)=\max_{j:j>i}\sum_{k=i+1}^j\frac{\pi([j,n])}{\nu(k-1,k)},\quad
 C_-(i)=\max_{j:j<i}\sum_{k=j}^{i-1}\frac{\pi([0,j])}{\nu(k,k+1)},
\end{equation}
where $\max\emptyset:=0$.

\begin{thm}\label{t-Hardy}
Let $G=(V,E)$ be a path on $\{0,1,...,n\}$ and $\pi,\nu$ be positive measures on $V,E$ with $\pi(V)=1$. Referring to \textnormal{(\ref{eq-c+c-})}, set $C(m)=\max\{C_+(m),C_-(m)\}$. Then, for $0\le m\le n$,
\[
 \frac{1}{4C(m)}\le\lambda^G_{\pi,\nu}\le \frac{1}{\min\{\pi([0,m]),\pi([m,n])\}C(m)}.
\]
In particular, if $M$ is a median of $\pi$, that is, $\pi([0,M])\ge 1/2$ and $\pi([M,n])\ge 1/2$, then
\[
 \frac{1}{4C(M)}\le\lambda^G_{\pi,\nu}\le \frac{2}{C(M)}.
\]
\end{thm}
\begin{rem}
Referring to the setting in Theorem \ref{t-Hardy}, the authors of \cite{CSal12} obtained $\lambda^G_{\pi,\nu}\ge 1/C'$, where
\[
 C'=\min_{0\le j\le n}\max\left\{\sum_{k=0}^{j-1}\frac{\pi([0,k])}{\nu(k,k+1)},\sum_{k=j+1}^{n}\frac{\pi([k,n])}{\nu(k-1,k)}\right\}.
\]
Theorem \ref{t-Hardy} indicates that  $1/C(M)$ is always of the same order as the spectral gap and provides an estimate that can be significantly better than $1/C'$.
\end{rem}

The proof of Theorem \ref{t-Hardy} is based on the following proposition, which is related to weighted Hardy's inequality on $\{1,...,n\}$.

\begin{prop}\label{p-Hardy}
Fix $n\ge 1$. Let $\mu,\pi$ be positive measures on $\{1,...,n\}$ and $A$ be the smallest constant such that
\begin{equation}\label{eq-Hardy}
 \sum_{i=1}^n\left(\sum_{j=1}^ig(j)\right)^2\pi(i)\le A\sum_{i=1}^ng^2(i)\mu(i),\quad\forall g\ne\mathbf{0}.
\end{equation}
Then, $B\le A\le 4B$, where
\[
 B=\max_{1\le i\le n}\left\{\pi([i,n])\sum_{j=1}^i\frac{1}{\mu(j)}\right\}.
\]
\end{prop}

\begin{rem}
Miclo \cite{M99} discussed the infinity case $\{1,2,...\}$ using the method in \cite{M72}, which was introduced by Muckenhoupt to study the continuous case $[0,\infty)$. For more information on the weighted Hardy inequality, see \cite{M99} and the references therein.
\end{rem}

\begin{proof}[Proof of Theorem \ref{t-Hardy}]
We first consider the lower bound of $\lambda^G_{\pi,\nu}$. Let $f$ be any function defined on $V$ and set $f_+=[f-f(m)]\mathbf{1}_{\{m,...,n\}}$ and $f_-=[f-f(m)]\mathbf{1}_{\{0,...,m\}}$. Then,
\begin{equation}\label{eq-lower}
 \frac{\mathcal{E}_\nu(f,f)}{\text{Var}_\pi(f)}\ge\frac{\mathcal{E}_\nu(f,f)}{\pi(f-f(m))^2}=\frac{\mathcal{E}_\nu(f_+,f_+)+\mathcal{E}_\nu(f_-,f_-)}{\pi(f_+^2)+\pi(f_-^2)}
\end{equation}
Set $g(j)=f(m+j)-f(m+j-1)$ for $1\le j\le n-m$ and $h(i)=f(m-i)-f(m-i+1)$ for $1\le i\le m$. Note that
\[
 \mathcal{E}_\nu(f_+,f_+)=\sum_{j=1}^{n-m}g^2(j)\nu(m+j-1,m+j),\,\pi(f_+^2)=\sum_{j=1}^{n-m}\left(\sum_{k=1}^jg(k)\right)^2\pi(m+j),
\]
and
\[
\mathcal{E}_\nu(f_-,f_-)=\sum_{i=1}^mh^2(i)\nu(m-i,m-i+1),\,\pi(f_-^2)=\sum_{j=1}^{m}\left(\sum_{k=1}^jh(k)\right)^2\pi(m-j).
\]
By Proposition \ref{p-Hardy}, the above computation implies that
\[
 \frac{\mathcal{E}_\nu(f_+,f_+)}{\pi(f_+^2)}\ge \frac{1}{4C_+(m)},\quad\frac{\mathcal{E}_\nu(f_-,f_-)}{\pi(f_-^2)}\ge \frac{1}{4C_-(m)}.
\]
Putting this back to (\ref{eq-lower}) gives the desired lower bound.

For the upper bound, we first consider the case $C=C_+(m)$. By Proposition \ref{p-Hardy}, $C_+(m)\le A$, where $A$ is the smallest constant $A$ such that, for any function $\phi$ defined on $\{1,2,...,n-m+1\}$,
\[
 \sum_{j=1}^{n-m}\left(\sum_{k=1}^j\phi(k)\right)^2\pi(m+j)\le A\sum_{j=1}^{n-m}\phi^2(j)\nu(m+j-1,m+j).
\]
Let $\phi$ be a minimizer for $A$, which must exist, and define $\psi$ by setting
\[
 \psi(i)=\begin{cases}\phi(1)+\cdots+\phi(i-m)&\text{for }m<i\le n\\0&\text{for }0\le i\le m\end{cases}.
\]
Clearly, $1/C_+(m)\ge 1/A=\mathcal{E}_\nu(\psi,\psi)/\pi(\psi^2)$. Without loss of generality, we may assume further that $\phi$ is nonnegative. Note that $\pi(\{\psi=0\})\ge\pi([0,m])$. By the Cauchy-Schwartz inequality, this implies $\pi(\psi)^2\le\pi(\{\psi>0\})\pi(\psi^2)\le\pi([m+1,n])\pi(\psi^2)$ and, then, $\text{Var}_\pi(\psi)\ge\pi([0,m])\pi(\psi^2)$. This leads to $1/C=1/C_+(m)\ge\pi([0,m])\lambda^G_{\pi,\nu}$. Similarly, if $C=C_-(m)$, one can prove that $1/C\ge\pi([m,n])\lambda^G_{\pi,\nu}$. This yields the upper bound of the spectral gap.
\end{proof}

\begin{proof}[Proof of Proposition \ref{p-Hardy}]
The proofs of Theorem \ref{t-Hardy} and Proposition \ref{p-Hardy} are very similar to those in \cite{M99}. Note that $A$ is attained at functions of the same sign and we assume that $g$ is non-negative. As $A$ is attainable, the minimizer $g$ for $A$ satisfies the following Euler-Lagrange equations.
\begin{equation}\label{eq-ELA}
 Ag(i)\mu(i)=\sum_{j=i}^n(g(1)+\cdots+g(j))\pi(j),\quad\forall 1\le i\le n.
\end{equation}
This is equivalent to the following system of equations.
\[
 A[g(i)\mu(i)-g(i+1)\mu(i+1)]=(g(1)+\cdots+g(i))\pi(i),\quad\forall 1\le i\le n,
\]
with the convention that $\mu(n+1):=0$. Inductively, one can show that $g>0$. Summing up (\ref{eq-ELA}) over $\{1,...,\ell\}$ yields
\begin{align}
 A\sum_{i=1}^\ell g(i)&=\sum_{i=1}^\ell\frac{1}{\mu(i)}\sum_{j=i}^n(g(1)+\cdots+g(j))\pi(j)\notag\\
&\ge\sum_{i=1}^\ell\sum_{j=\ell}^n\frac{(g(1)+\cdots+g(j))\pi(j)}{\mu(i)}\notag\\&\ge\left(\sum_{i=1}^\ell g(i)\right)\left(\sum_{i=1}^\ell\frac{1}{\mu(i)}\right)\pi([\ell,n]).\notag
\end{align}
This leads to $A\ge B$.

To see the upper bound, we use Miclo's method in \cite{M99}. Set $N(j)=\sum_{i=1}^j1/\mu(i)$. By the Cauchy inequality, the left side of (\ref{eq-Hardy}) is bounded above by
\[
 \sum_{i=1}^n\pi(i)\sum_{j=1}^ig^2(j)\mu(j)N^{1/2}(j)\sum_{l=1}^i\frac{1}{\mu(l)N^{1/2}(l)}.
\]
Note that, for $s>0,t>0$, $t^{1/2}-s^{1/2}\ge(t-s)/(2t^{1/2})$. This implies $2(N^{1/2}(l)-N^{1/2}(l-1))\ge 1/(\mu(l) N^{1/2}(l))$ with the convention that $N(0):=0$. Consequently, we have
\[
\sum_{l=1}^i\frac{1}{\mu(l)N^{1/2}(l)}\le 2N^{1/2}(i)\le \left(\frac{4B}{\pi([i,n])}\right)^{1/2},
\]
and, thus,
\begin{align}
  \sum_{i=1}^n\left(\sum_{j=1}^ig(j)\right)^2\pi(i)&\le\sqrt{4B}\sum_{i=1}^n\frac{\pi(i)}{\pi([i,n])^{1/2}}
\sum_{j=1}^ig^2(j)\mu(j)N^{1/2}(j)\notag\\
&\le\sqrt{4B}\sum_{j=1}^ng^2(j)\mu(j)N^{1/2}(j)\sum_{i=j}^n\frac{\pi(i)}{\pi([i,n])^{1/2}}\notag.
\end{align}
Again, the inequality for $s,t$ implies
\[
 \sum_{i=j}^n\frac{\pi(i)}{\pi([i,n])^{1/2}}\le 2\pi([j,n])^{1/2}\le\frac{\sqrt{4B}}{N^{1/2}(j)}.
\]
This gives the desired upper bound.
\end{proof}

Next, we consider a special case. Let $\pi,\nu$ are measures on $V=\{0,1,...,n\},E=\{\{i,i+1\}|0\le i<n\}$ with $\pi(V)=1$. Suppose
\begin{equation}\label{eq-sym}
 \pi(i)=\pi(n-i),\quad\nu(i,i+1)=\nu(n-i-1,n-i),\quad \forall 0\le i\le n/2.
\end{equation}
By the symmetry of $\pi$ and $\nu$, if $\psi$ is a minimizer for $\lambda^G_{\pi,\nu}$ with $\pi(\psi)=0$, then $\psi$ is either symmetric or anti-symmetric at $n/2$. The former is set aside because $\psi$ is known to be monotonic and this leads to the case $\psi(n-i)=-\psi(i)$ for $0\le i\le n/2$. If $n$ is even with $n=2k$, then $\psi(k)=0$ and this implies
\[
 \lambda^G_{\pi,\nu}=\inf\left\{\frac{\sum_{i=1}^k(f(i)-f(i-1))^2\nu(i-1,i)}{\sum_{i=0}^{k-1}f^2(i)\pi(i)}\bigg|f(k)=0,f\ne\mathbf{0}\right\}.
\]
Equivalently, if one sets $g(i)=f(k-i)-f(k-i+1)$ and $\mu(i)=\nu(k-i,k-i+1)$ for $1\le i\le k$, then $1/\lambda^G_{\pi,\nu}$ is the smallest constant $A$ such that
\begin{equation}\label{eq-Hardy2}
 \sum_{i=1}^k\left(\sum_{j=1}^ig(j)\right)^2\pi(k-i)\le A\sum_{i=1}^kg^2(i)\mu(i),\quad\forall g\ne\mathbf{0}.
\end{equation}
Similarly, if $n$ is odd with $n=2k-1$, one has
\[
 \lambda^G_{\pi,\nu}=\min\left\{\frac{\sum_{i=1}^{k-1}(f(i)-f(i-1))^2\nu(i-1,i)+2f^2(k-1)\nu(k-1,k)}{\sum_{i=0}^{k-1}f^2(i)\pi(i)}\bigg|f\ne\mathbf{0}\right\},
\]
and this leads to (\ref{eq-Hardy2}) with $g(1)=f(k-1)$, $\mu(1)=2\nu(k-1,k)$ and, for $2\le i\le k$, $g(i)=f(k-i)-f(k-i+1)$ and $\mu(i)=\nu(k-i,k-i+1)$. A direct application of Proposition \ref{p-Hardy} implies the following theorem.

\begin{thm}\label{t-sym1}
Let $G=(V,E)$ be the graph with $V=\{0,1,...,n\}$, $E=\{\{i,i+1\}|i=0\le i<n\}$ and let $\pi,\nu$ be positive measures on $V,E$ satisfying $\pi(V)=1$ and \textnormal{(\ref{eq-sym})}. Set $N=\lceil n/2\rceil$. Then, $1/(4C)\le\lambda^G_{\pi,\nu}\le 1/C$, where
\[
 C=\max_{0\le i<N}\left\{\pi([0,i])\sum_{j=i}^{N-1}\frac{1}{\nu(j,j+1)}\right\}\quad\text{if $n$ is even},
\]
and
\[
 C=\max_{0\le i<N}\left\{\pi([0,i])\left(\sum_{j=i}^{N-2}\frac{1}{\nu(j,j+1)}+\frac{1}{2\nu(N-1,N)}\right)\right\}\quad\text{if $n$ is odd}.
\]
\end{thm}

\begin{rem}
The symmetry of $\pi,\nu$ in Theorems \ref{t-sym1} can be relaxed using the comparison technique.
\end{rem}

\bibliographystyle{plain}
\bibliography{reference}

\begin{thebibliography}{10}

\bibitem{A83}
David Aldous.
\newblock Random walks on finite groups and rapidly mixing {M}arkov chains.
\newblock In {\em Seminar on probability, XVII}, volume 986 of {\em Lecture
  Notes in Math.}, pages 243--297. Springer, Berlin, 1983.

\bibitem{BBF09}
J.~Barrera, O.~Bertoncini, and R.~Fern{\'a}ndez.
\newblock Abrupt convergence and escape behavior for birth and death chains.
\newblock {\em J. Stat. Phys.}, 137(4):595--623, 2009.

\bibitem{BS87}
M.~Brown and Y.-S. Shao.
\newblock Identifying coefficients in the spectral representation for first
  passage time distributions.
\newblock {\em Probab. Engrg. Inform. Sci.}, 1:69--74, 1987.

\bibitem{GY-PhD}
Guan-Yu Chen.
\newblock {\em The cutoff phenomenon for finite Markov chains}.
\newblock PhD thesis, Cornell University, 2006.

\bibitem{CSal08}
Guan-Yu Chen and Laurent Saloff-Coste.
\newblock The cutoff phenomenon for ergodic markov processes.
\newblock {\em Electron. J. Probab.}, 13:26--78, 2008.

\bibitem{CSal12-2}
Guan-Yu Chen and Laurent Saloff-Coste.
\newblock Comparison of cutoffs between lazy walks and markovian semigroups.
\newblock In preparation, 2012.

\bibitem{CSal12}
Guan-Yu Chen and Laurent Saloff-Coste.
\newblock Spectral computations for birth and death chains.
\newblock In preparation, 2012.

\bibitem{DS98}
P.~Diaconis and L.~Saloff-Coste.
\newblock What do we know about the {M}etropolis algorithm?
\newblock {\em J. Comput. System Sci.}, 57(1):20--36, 1998.
\newblock 27th Annual ACM Symposium on the Theory of Computing (STOC'95) (Las
  Vegas, NV).

\bibitem{D88}
Persi Diaconis.
\newblock {\em Group representations in probability and statistics}.
\newblock Institute of Mathematical Statistics Lecture Notes---Monograph
  Series, 11. Institute of Mathematical Statistics, Hayward, CA, 1988.

\bibitem{D96cutoff}
Persi Diaconis.
\newblock The cutoff phenomenon in finite {M}arkov chains.
\newblock {\em Proc. Nat. Acad. Sci. U.S.A.}, 93(4):1659--1664, 1996.

\bibitem{DS93-1}
Persi Diaconis and Laurent Saloff-Coste.
\newblock Comparison techniques for random walk on finite groups.
\newblock {\em Ann. Probab.}, 21(4):2131--2156, 1993.

\bibitem{DS93-2}
Persi Diaconis and Laurent Saloff-Coste.
\newblock Comparison theorems for reversible {M}arkov chains.
\newblock {\em Ann. Appl. Probab.}, 3(3):696--730, 1993.

\bibitem{DS06}
Persi Diaconis and Laurent Saloff-Coste.
\newblock Separation cut-offs for birth and death chains.
\newblock {\em Ann. Appl. Probab.}, 16(4):2098--2122, 2006.

\bibitem{DS91}
Persi Diaconis and Daniel Stroock.
\newblock Geometric bounds for eigenvalues of {M}arkov chains.
\newblock {\em Ann. Appl. Probab.}, 1(1):36--61, 1991.

\bibitem{DLP10}
Jian Ding, Eyal Lubetzky, and Yuval Peres.
\newblock Total variation cutoff in birth-and-death chains.
\newblock {\em Probab. Theory Related Fields}, 146(1-2):61--85, 2010.

\bibitem{M99}
L.~Miclo.
\newblock An example of application of discrete {H}ardy's inequalities.
\newblock {\em Markov Process. Related Fields}, 5(3):319--330, 1999.

\bibitem{M72}
Benjamin Muckenhoupt.
\newblock Hardy's inequality with weights.
\newblock {\em Studia Math.}, 44:31--38, 1972.
\newblock Collection of articles honoring the completion by Antoni Zygmund of
  50 years of scientific activity, I.

\bibitem{SC99}
L.~Saloff-Coste.
\newblock Simple examples of the use of {N}ash inequalities for finite {M}arkov
  chains.
\newblock In {\em Stochastic geometry ({T}oulouse, 1996)}, volume~80 of {\em
  Monogr. Statist. Appl. Probab.}, pages 365--400. Chapman \& Hall/CRC, Boca
  Raton, FL, 1999.

\end{thebibliography}

\end{document}